\documentclass[12pt]{amsart}

\usepackage{times}
\usepackage{amsmath,amsfonts,amssymb,amsopn,amscd,amsthm}
\usepackage{mathbbol}
\usepackage{mathrsfs}
\usepackage{graphicx,xspace}
\usepackage{epsfig}
\usepackage{enumerate} 
\usepackage{enumitem}
\usepackage{xcolor}
\usepackage[all]{xypic}
\usepackage{pgf,tikz}
\usepackage{mathrsfs}
\usetikzlibrary{arrows}
\definecolor{sqsqsq}{rgb}{0.12549019607843137,0.12549019607843137,0.12549019607843137}
\definecolor{zzttqq}{rgb}{0.6,0.2,0}
\definecolor{uququq}{rgb}{0.25098039215686274,0.25098039215686274,0.25098039215686274}
\definecolor{ffffff}{rgb}{1,1,1}

\usepackage[T1]{fontenc}
\usepackage[utf8]{inputenc}

\usepackage{hyperref}
\usepackage{psfrag}
\usepackage{bm}
\usepackage{youngtab}
\usepackage{tikz}
\usetikzlibrary{snakes}
\usepackage[all]{xy}
\usepackage{varioref}
\theoremstyle{plain}
\newtheorem{theoreme}{Theorem}[section]

\newtheorem{lem}[theoreme]{Lemma}

\theoremstyle{remark}

\newtheorem{ex}[theoreme]{Example}

\newtheorem{c-ex}[theoreme]{Counter-example}

\newcommand{\field}{\mathbb{C}}
\date{}
\def\pr{\ast}
\DeclareUnicodeCharacter{03C7}{\chi}
\DeclareUnicodeCharacter{B2}{^{2}}
\DeclareMathOperator{\supp}{supp}
\DeclareMathOperator{\type-cyclique}{type-cyclique}
\DeclareMathOperator{\coset-type}{coset-type}
\DeclareMathOperator{\set-type}{set-type}
\DeclareMathOperator{\Proj}{Proj}
\DeclareMathOperator{\ch}{ch}
\DeclareMathOperator{\ct}{ct}
\DeclareMathOperator{\ty}{type}
\DeclareMathOperator{\SC}{SC}
\DeclareMathOperator{\NCSym}{NCSym}
\DeclareMathOperator{\I}{I}
\DeclareMathOperator{\tr}{tr}
\DeclareMathOperator{\sign}{sign}
\DeclareMathOperator{\Fix}{Fix}
\DeclareMathOperator{\Ind}{Ind}
\DeclareMathOperator{\Res}{Res}
\DeclareMathOperator{\End}{End}
\DeclareMathOperator{\Mat}{Mat}
\DeclareMathOperator{\vect}{Vect}
\DeclareMathOperator{\diag}{diag}
\DeclareMathOperator{\Id}{Id}
\DeclareMathOperator{\id}{id}
\newcommand{\Hecke}{$\mathbb{C}[\mathcal{B}_n\setminus \mathcal{S}_{2n}/\mathcal{B}_n]$}
\newcommand{\hecke}{$(\mathcal{S}_{kn},\mathcal{B}_{kn}^k)$ }
\author[O. Tout]{Omar Tout}
\address{Department of Mathematics, College of Science, Sultan Qaboos University, P. O Box 36, Al Khod 123, Sultanate of Oman}
\email{o.tout@squ.edu.om}

\title[Domination number through space projections]{On the domination number of the cartesian product of\\ the path graph and any pair of graphs}

\begin{document}
\maketitle

\begin{abstract} It is known that for any graph $G,$ $\gamma (G\square P_2)\geq \gamma (G)$ where $\gamma$ stands for the domination number, $\square$ for the cartesian product and $P_2$ is the path graph on two vertices. In an attempt to prove Vizing's conjecture, Clark and Suen proved in $2000$ that $\gamma (X\square Y)\geq \frac{1}{2}\gamma (X)\gamma (Y)$ for any pair of graphs $X$ and $Y.$ Combining these two inequalities, we have $\gamma (X\square Y\square P_2)\geq \frac{1}{2}\gamma (X)\gamma (Y).$ In this paper, we use space projections to improve this lower bound and show that $\gamma (X\square Y\square P_2)\geq \frac{2}{3}\gamma (X)\gamma (Y)$ for any pair of graphs $X$ and $Y.$ In addition, we prove that $\gamma (X\square Y\square P_{n})\geq c_n\gamma (X)\gamma (Y)\gamma (P_{n}),$ where $c_n$ is almost $\frac{3}{4}$ when $n$ is big enough. 
\end{abstract}

\bigskip

%\section{Introduction}

Let $G$ be a simple finite graph and let $V(G)$ be its set of vertices. We say that a vertex $u\in V (G)$ dominates a vertex $v$ if $u=v$ or $v$ is adjacent to $u.$ A dominating set of $G,$ is a subset $S$ of $V(G)$ whose vertices dominate all the vertices of $G.$ The domination number of $G,$ denoted $\gamma (G),$ is the size of a smallest dominating set of $G.$ 

The Cartesian product $X\square Y$ of two graphs $X$ and $Y$ is the graph whose vertex set is $V(X)\times V(Y)$ and edge set defined as follows. Two vertices $(x_1,y_1)$ and $(x_2,y_2)$ are adjacent in $X\square Y$ if either $x_1=x_2$ and $y_1$ and $y_2$ are adjacent in $Y,$ or $y_1=y_2$ and $x_1$ and $x_2$ are adjacent in $X.$ By definition, the Cartesian product of graphs is commutative, in the sence that $X\square Y$ is isomorphic to $Y\square X.$ For $y\in V(Y),$ the subgraph of $X\square Y$ induced by $\lbrace (x,y)\mid x\in V(X)\rbrace,$ is called an $X$-fiber and denoted $X^y.$

It is well known, see \cite{6} for example, that for any graph $G,$ $\gamma (G\square P_2)\geq \gamma (G)$ where $P_2$ is the path graph on two vertices. In an attempt to prove Vizing's conjecture, which appeared in \cite{1}, Clark and Suen proved in \cite{2} that $\gamma (X\square Y)\geq \frac{1}{2}\gamma (X)\gamma (Y)$ for any pair of graphs $X$ and $Y.$ This lower bound was improved by Suen and Tarr in \cite{3}, and recently by Zerbib in \cite{4}. However, $\frac{1}{2}$ remains the best coefficient obtained toward proving Vizing's conjecture. In Theorem \ref{Main Theorem}, we show that $\gamma (X\square Y\square P_2)\geq \frac{2}{3}\gamma (X)\gamma (Y).$ To prove this inequality, we use space projections and we follow the new framework to approach Vizing's conjecture developed in \cite{5}. Then in Theorem \ref{Main Theorem 2}, we prove that for any $n\geq 1,$ $\gamma (X\square Y\square P_{n})\geq c_n\gamma (X)\gamma (Y)\gamma (P_{n})$ where $c_n$ is almost $\frac{3}{4}$ when $n$ is big enough.

%\section{The domination number and space projections}

\bigskip

For a natural number $k,$ the set $\lbrace 1,\cdots ,k\rbrace$ will be denote by $[k].$ Let $X$ be a fixed graph with $\gamma (X)=k.$ Consider $\lbrace u_1,\cdots ,u_k\rbrace$ to be a minimum dominating set of $X.$ Let $\lbrace \pi_{i}; i\in [k]\rbrace$ be a partition of $V(X)$ chosen so that $u_i\in \pi_{i}$ and $\pi_{i}\subseteq N[u_i]$ for any $i\in [k],$ where $N[x]$ is the set containing $x$ and all the vertices adjacent to it in $X.$ Let $Y$ and $Z$ be two arbitrary graphs. In the cartesian product $X\square Y\square Z,$ for any $i\in [k]$ and any $(y,z)\in V(Y)\times V(Z),$ let $\pi_{i}^{y,z}$ be the following set of vertices
\[
\pi_{i}^{y,z}:=\lbrace(a,y,z); a\in \pi_i \rbrace.
\]
In what follows, the sets $\pi_{i}^{y,z}$ will be called cells. It would be clear that the set of all cells make a partition for $V(X\square Y\square Z).$

Now let $D$ be a minimum dominating set of $X\square Y\square Z.$ A vertex $(a,b,c)$ in $X\square Y\square Z$ is said to be $X$-dominated (resp. $Y$-dominated, $Z$-dominated) if $(a,b,c)$ is dominated by a vertex $(a',b,c)\in D$ (resp. $(a,b',c)\in D,$ $(a,b,c')\in D$). We color the cells $\pi_{i}^{y,z}$ with eight colors as follows:
\begin{enumerate}
\item[•] The cell $\pi_{i}^{y,z}$ is blue if $D\cap \pi_{i}^{y,z}\neq \emptyset$ and none of its vertices are $Y$-dominated or $Z$-dominated (then all of its vertices are exclusively $X$-dominated).
\item[•] The cell $\pi_{i}^{y,z}$ is green if $D\cap \pi_{i}^{y,z}\neq \emptyset,$ is not blue, and none of its vertices are $Z$-dominated (then all of its vertices are exclusively $X$-dominated and $Y$-dominated).
\item[•] The cell $\pi_{i}^{y,z}$ is yellow if $D\cap \pi_{i}^{y,z}\neq \emptyset,$ is not blue, and none of its vertices are $Y$-dominated (then all of its vertices are exclusively $X$-dominated and $Z$-dominated).
\item[•] The cell $\pi_{i}^{y,z}$ is orange if $D\cap \pi_{i}^{y,z}\neq \emptyset$ and is not blue, green or yellow.
\item[•] The cell $\pi_{i}^{y,z}$ is red if $D\cap \pi_{i}^{y,z}= \emptyset$ and none of its vertices are $Y$-dominated or $Z$-dominated.
\item[•] The cell $\pi_{i}^{y,z}$ is pink if $D\cap \pi_{i}^{y,z}= \emptyset,$ is not red, and none of its vertices are $Z$-dominated.
\item[•] The cell $\pi_{i}^{y,z}$ is maroon if $D\cap \pi_{i}^{y,z}= \emptyset,$ is not red, and none of its vertices are $Y$-dominated.
\item[•] All the remaining cells are colored white.
\end{enumerate}

\begin{ex}\label{Main example} Consider the three graphs $X, Y$ and $Z$ with vertex sets \newline $V(X)=\lbrace x_1,x_2,x_3,x_4,x_5,x_6,x_7,x_8\rbrace,$ $V(Y)=\lbrace y_1,y_2,y_3,y_4\rbrace,$ $V(Z)=\lbrace z_1,z_2\rbrace$ respectively and edge sets defined as follows 
$$E(X)=\lbrace x_1x_2,x_2x_3,x_2x_6,x_3x_4,x_4x_5,x_5x_6,x_5x_8,x_6x_7,x_7x_8\rbrace,$$ 
$$E(Y)=\lbrace y_1y_2,y_1y_4,y_2y_3,y_3y_4\rbrace,$$ 
and 
$$E(Z)=\lbrace z_1z_2\rbrace.$$ 
It would be easy to check that the domination number of $X$ is $3,$ of $Y$ is $2$ and that of $Z$ is $1.$ Consider the dominating set $\lbrace x_2,x_5,x_7\rbrace$ of $X.$ Let $\lbrace \pi_{1}, \pi_{2},\pi_{3}\rbrace$ be the partition of $V(X)$ where:
$$\pi_{1}=\lbrace x_1,x_2,x_3\rbrace,$$
$$\pi_{2}=\lbrace x_4,x_5\rbrace$$
and
$$\pi_{3}=\lbrace x_6,x_7,x_8\rbrace.$$

We used SageMath \cite{7} to verify that the domination number of $X\square Y\square Z$ is $13$ and that a minimal dominating set is given by 
$$D=\lbrace (x_1, y_4, z_1),
 (x_2, y_2,z_1),
 (x_4, y_4,z_1),
 (x_5, y_2,z_1),
 (x_7, y_1,z_1),
 (x_7, y_3,z_1),
 (x_8, y_4,z_1),$$
 $$(x_1, y_2,z_2),
 (x_3, y_1,z_2),
 (x_3, y_3,z_2),
(x_5, y_2,z_2),
 (x_6, y_4,z_2), 
 (x_8, y_2,z_2)\rbrace.$$

In Figure \ref{fig:main fig}, we show the coloring of the cells with eight colors as well as the coloring of the vertices of $D$ with the four colors blue, green, yellow and orange. For visibility, we used dashed lines for white colored cells. The cells in the bottom correspond to the layer $(X\square Y)^{z_1}$ while on the top figures the cells of the layer $(X\square Y)^{z_2}.$ We omit the edges in the cartesian product $X\square Y\square Z$ to make the picture easily readable. We strongly recommend the reader to avoid using black and white printing for Figure \ref{fig:main fig} and use instead color printing in order to obtain a clear, not confusing, picture.

The reader is invited to check that the colors given to cells agrees with our construction. For example, the cell $\pi_{1}^{y_4,z_1}$ is colored blue since it contains a vertex of $D$ and all its vertices are $X$-dominated and no one of them is $Y$ or $Z$-dominated. The cell $\pi_{3}^{y_1,z_1}$ is colored green since it contains a vertex of $D$ and its vertex $(x_8,y_1,z_1)$ is $Y$-dominated by $(x_8,y_4,z_1)$ and none of its vertices is $Z$-dominated. Also, $\pi_{1}^{y_2,z_1}$ is colored yellow because it contains vertices of $D$ and its vertices are $X$-dominated or $Z$-dominated and none of them is $Y$-dominated while $\pi_{1}^{y_2,z_2}$ is colored orange because it contains a vertex of $D$ and it has vertices $X$-dominated, $Y$-dominated and $Z$-dominated. In addition, $\pi_{1}^{y_1,z_1}$ is colored white (dashed for visibility) because it contains no vertex of $D$ and it has vertices $Y$-dominated and vertices $Z$-dominated. The cell $\pi_{2}^{y_1,z_1}$ is colored pink because it contains no vertex of $D$ and its vertices are $X$ or $Y$-dominated and none of them is $Z$-dominated. Finally, $\pi_{2}^{y_4,z_2}$ is colored maroon because it contains no vertex of $D$ and its vertices are $X$ or $Z$-dominated and none of them is $Y$-dominated. The provided example lacks red cells.

\definecolor{zzttqq}{rgb}{0.6,0.2,0}
\definecolor{fuqqzz}{rgb}{0.9568627450980393,0,0.6}
\definecolor{cqcqcq}{rgb}{0.7529411764705882,0.7529411764705882,0.7529411764705882}
\definecolor{ffwwqq}{rgb}{1,0.4,0}
\definecolor{qqffqq}{rgb}{0,1,0}
\definecolor{ffffww}{rgb}{1,1,0.4}
\definecolor{ffffff}{rgb}{1,1,1}
\definecolor{qqqqff}{rgb}{0,0,1}

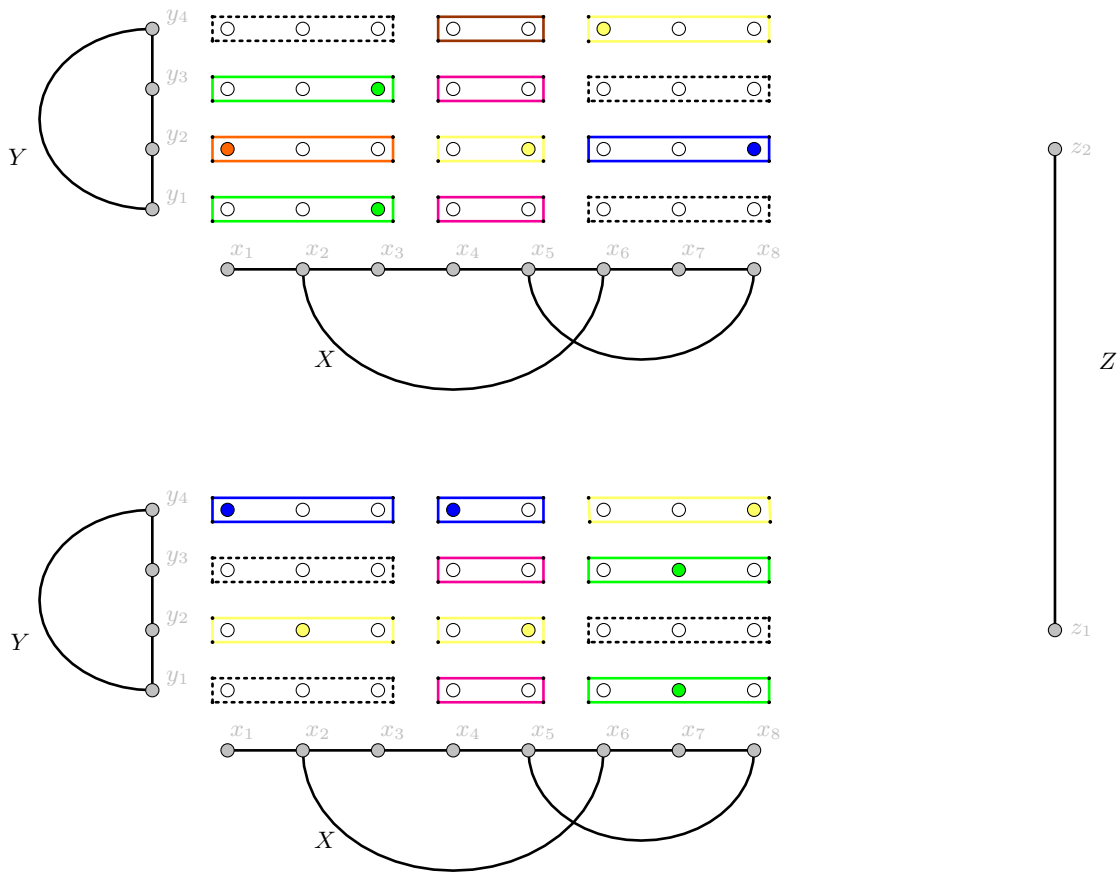
\begin{figure}[htbp]
\begin{center}
\begin{tikzpicture}[line cap=round,line join=round,>=triangle 45,x=1cm,y=.8cm]
%\clip(-24.490023789873916,-2.6909454977401133) rectangle (3.791886061548467,14.747451832172462);
\draw [line width=1pt] (-14,1)-- (-13,1);
\draw [line width=1pt] (-13,1)-- (-12,1);
\draw [line width=1pt] (-12,1)-- (-11,1);
\draw [line width=1pt] (-11,1)-- (-10,1);
\draw [line width=1pt] (-10,1)-- (-9,1);
\draw [line width=1pt] (-9,1)-- (-8,1);
\draw [line width=1pt] (-8,1)-- (-7,1);
\draw [line width=1pt] (-15,2)-- (-15,3);
\draw [line width=1pt] (-15,3)-- (-15,4);
\draw [line width=1pt] (-15,4)-- (-15,5);
\draw [line width=1pt] (-14,9)-- (-13,9);
\draw [line width=1pt] (-13,9)-- (-12,9);
\draw [line width=1pt] (-12,9)-- (-11,9);
\draw [line width=1pt] (-11,9)-- (-10,9);
\draw [line width=1pt] (-10,9)-- (-9,9);
\draw [line width=1pt] (-9,9)-- (-8,9);
\draw [line width=1pt] (-8,9)-- (-7,9);
\draw [line width=1pt] (-15,10)-- (-15,11);
\draw [line width=1pt] (-15,11)-- (-15,12);
\draw [line width=1pt] (-15,12)-- (-15,13);
\draw [line width=1pt] (-3,3)-- (-3,11);
\draw [shift={(-8.5,9)},line width=1pt]  plot[domain=3.141592653589793:6.283185307179586,variable=\t]({1*1.5*cos(\t r)+0*1.5*sin(\t r)},{0*1.5*cos(\t r)+1*1.5*sin(\t r)});
\draw [shift={(-11,9)},line width=1pt]  plot[domain=3.141592653589793:6.283185307179586,variable=\t]({1*2*cos(\t r)+0*2*sin(\t r)},{0*2*cos(\t r)+1*2*sin(\t r)});
\draw [shift={(-15,11.5)},line width=1pt]  plot[domain=1.5707963267948966:4.71238898038469,variable=\t]({1*1.5*cos(\t r)+0*1.5*sin(\t r)},{0*1.5*cos(\t r)+1*1.5*sin(\t r)});
\draw [shift={(-8.5,1)},line width=1pt]  plot[domain=3.141592653589793:6.283185307179586,variable=\t]({1*1.5*cos(\t r)+0*1.5*sin(\t r)},{0*1.5*cos(\t r)+1*1.5*sin(\t r)});
\draw [shift={(-11,1)},line width=1pt]  plot[domain=3.141592653589793:6.283185307179586,variable=\t]({1*2*cos(\t r)+0*2*sin(\t r)},{0*2*cos(\t r)+1*2*sin(\t r)});
\draw [shift={(-15,3.5)},line width=1pt]  plot[domain=1.5707963267948966:4.71238898038469,variable=\t]({1*1.5*cos(\t r)+0*1.5*sin(\t r)},{0*1.5*cos(\t r)+1*1.5*sin(\t r)});
\draw [line width=1pt,color=qqffqq] (-14.2,9.8)-- (-11.8,9.8);
\draw [line width=1pt,color=qqffqq] (-11.8,9.8)-- (-11.8,10.2);
\draw [line width=1pt,color=qqffqq] (-11.8,10.2)-- (-14.2,10.2);
\draw [line width=1pt,color=qqffqq] (-14.2,10.2)-- (-14.2,9.8);
\draw [line width=1pt,color=fuqqzz] (-11.2,10.2)-- (-9.8,10.2);
\draw [line width=1pt,color=fuqqzz] (-9.8,10.2)-- (-9.8,9.8);
\draw [line width=1pt,color=fuqqzz] (-9.8,9.8)-- (-11.2,9.8);
\draw [line width=1pt,color=fuqqzz] (-11.2,9.8)-- (-11.2,10.2);
\draw [line width=1pt,dotted] (-9.2,10.2)-- (-6.8,10.2);
\draw [line width=1pt,dotted] (-6.8,10.2)-- (-6.8,9.8);
\draw [line width=1pt,dotted] (-6.8,9.8)-- (-9.2,9.8);
\draw [line width=1pt,dotted] (-9.2,9.8)-- (-9.2,10.2);
\draw [line width=1pt,color=zzttqq] (-11.2,13.2)-- (-9.8,13.2);
\draw [line width=1pt,color=zzttqq] (-9.8,13.2)-- (-9.8,12.8);
\draw [line width=1pt,color=zzttqq] (-9.8,12.8)-- (-11.2,12.8);
\draw [line width=1pt,color=zzttqq] (-11.2,12.8)-- (-11.2,13.2);
\draw [line width=1pt,color=ffffww] (-9.2,13.2)-- (-6.8,13.2);
\draw [line width=1pt,color=ffffww] (-6.8,13.2)-- (-6.8,12.790919741852175);
\draw [line width=1pt,color=ffffww] (-6.8,12.790919741852175)-- (-9.2,12.790919741852175);
\draw [line width=1pt,color=ffffww] (-9.2,12.790919741852175)-- (-9.2,13.2);
\draw [line width=1pt,dotted] (-14.2,13.2)-- (-11.8,13.2);
\draw [line width=1pt,dotted] (-11.8,13.2)-- (-11.8,12.8);
\draw [line width=1pt,dotted] (-11.8,12.8)-- (-14.2,12.8);
\draw [line width=1pt,dotted] (-14.2,12.8)-- (-14.2,13.2);
\draw [line width=1pt,color=qqqqff] (-14.2,5.2)-- (-11.8,5.2);
\draw [line width=1pt,color=qqqqff] (-11.8,5.2)-- (-11.8,4.8);
\draw [line width=1pt,color=qqqqff] (-11.8,4.8)-- (-14.2,4.8);
\draw [line width=1pt,color=qqqqff] (-14.2,4.8)-- (-14.2,5.2);
\draw [line width=1pt,color=qqqqff] (-11.2,5.2)-- (-9.8,5.2);
\draw [line width=1pt,color=qqqqff] (-9.8,5.2)-- (-9.8,4.8);
\draw [line width=1pt,color=qqqqff] (-9.8,4.8)-- (-11.2,4.8);
\draw [line width=1pt,color=qqqqff] (-11.2,4.8)-- (-11.2,5.2);
\draw [line width=1pt,color=ffffww] (-9.2,5.2)-- (-6.8,5.2);
\draw [line width=1pt,color=ffffww] (-6.8,5.2)-- (-6.78484895589595,4.799937821734619);
\draw [line width=1pt,color=ffffww] (-9.18484895589595,4.799937821734619)-- (-9.2,5.2);
\draw [line width=1pt,color=ffffww] (-9.18484895589595,4.799937821734619)-- (-6.78484895589595,4.799937821734619);
\draw [line width=1pt,color=fuqqzz] (-11.2,2.2)-- (-9.8,2.2);
\draw [line width=1pt,color=fuqqzz] (-9.8,2.2)-- (-9.8,1.8);
\draw [line width=1pt,color=fuqqzz] (-9.8,1.8)-- (-11.2,1.8);
\draw [line width=1pt,color=fuqqzz] (-11.2,1.8)-- (-11.2,2.2);
\draw [line width=1pt,color=qqffqq] (-9.2,2.2)-- (-6.8,2.2);
\draw [line width=1pt,color=qqffqq] (-6.8,2.2)-- (-6.8,1.8);
\draw [line width=1pt,color=qqffqq] (-6.8,1.8)-- (-9.2,1.8);
\draw [line width=1pt,color=qqffqq] (-9.2,1.8)-- (-9.2,2.2);
\draw [line width=1pt,dotted] (-14.2,2.2)-- (-11.8,2.2);
\draw [line width=1pt,dotted] (-11.8,2.2)-- (-11.8,1.8);
\draw [line width=1pt,dotted] (-11.8,1.8)-- (-14.2,1.8);
\draw [line width=1pt,dotted] (-14.2,1.8)-- (-14.2,2.2);
\draw [line width=1pt,color=ffffww] (-14.2,3.2)-- (-11.8,3.2);
\draw [line width=1pt,color=ffffww] (-11.8,3.2)-- (-11.8,2.8);
\draw [line width=1pt,color=ffffww] (-11.8,2.8)-- (-14.2,2.8);
\draw [line width=1pt,color=ffffww] (-14.2,2.8)-- (-14.2,3.2);
\draw [line width=1pt,color=ffffww] (-11.2,3.2)-- (-9.8,3.2);
\draw [line width=1pt,color=ffffww] (-9.8,3.2)-- (-9.8,2.8);
\draw [line width=1pt,color=ffffww] (-9.8,2.8)-- (-11.2,2.8);
\draw [line width=1pt,color=ffffww] (-11.2,2.8)-- (-11.2,3.2);
\draw [line width=1pt,color=qqffqq] (-9.2,4.2)-- (-6.8,4.2);
\draw [line width=1pt,color=qqffqq] (-6.8,4.2)-- (-6.8,3.8);
\draw [line width=1pt,color=qqffqq] (-6.8,3.8)-- (-9.2,3.8);
\draw [line width=1pt,color=qqffqq] (-9.2,3.8)-- (-9.2,4.2);
\draw [line width=1pt,color=qqffqq] (-14.2,12.2)-- (-11.8,12.2);
\draw [line width=1pt,color=qqffqq] (-11.8,12.2)-- (-11.8,11.8);
\draw [line width=1pt,color=qqffqq] (-11.8,11.8)-- (-14.2,11.8);
\draw [line width=1pt,color=qqffqq] (-14.2,11.8)-- (-14.2,12.2);
\draw [line width=1pt,color=ffwwqq] (-14.2,11.2)-- (-11.8,11.2);
\draw [line width=1pt,color=ffwwqq] (-11.8,11.2)-- (-11.8,10.8);
\draw [line width=1pt,color=ffwwqq] (-11.8,10.8)-- (-14.2,10.8);
\draw [line width=1pt,color=ffwwqq] (-14.2,10.8)-- (-14.2,11.2);
\draw [line width=1pt,color=ffffww] (-11.2,11.2)-- (-9.8,11.2);
\draw [line width=1pt,color=ffffww] (-9.8,11.2)-- (-9.8,10.8);
\draw [line width=1pt,color=ffffww] (-9.8,10.8)-- (-11.2,10.8);
\draw [line width=1pt,color=ffffww] (-11.2,10.8)-- (-11.2,11.2);
\draw [line width=1pt,color=qqqqff] (-9.2,11.2)-- (-6.8,11.2);
\draw [line width=1pt,color=qqqqff] (-6.8,11.2)-- (-6.8,10.8);
\draw [line width=1pt,color=qqqqff] (-6.8,10.8)-- (-9.2,10.8);
\draw [line width=1pt,color=qqqqff] (-9.2,10.8)-- (-9.2,11.2);
\draw [line width=1pt,dotted] (-9.2,3.2)-- (-6.8,3.2);
\draw [line width=1pt,dotted] (-6.8,3.2)-- (-6.8,2.8);
\draw [line width=1pt,dotted] (-6.8,2.8)-- (-9.2,2.8);
\draw [line width=1pt,dotted] (-9.2,2.8)-- (-9.2,3.2);
\draw [line width=1pt,dotted] (-14.2,4.2)-- (-11.8,4.2);
\draw [line width=1pt,dotted] (-11.8,4.2)-- (-11.8,3.8);
\draw [line width=1pt,dotted] (-11.8,3.8)-- (-14.2,3.8);
\draw [line width=1pt,dotted] (-14.2,3.8)-- (-14.2,4.2);
\draw [line width=1pt,color=fuqqzz] (-11.2,4.2)-- (-9.8,4.2);
\draw [line width=1pt,color=fuqqzz] (-9.8,4.2)-- (-9.8,3.8);
\draw [line width=1pt,color=fuqqzz] (-9.8,3.8)-- (-11.2,3.8);
\draw [line width=1pt,color=fuqqzz] (-11.2,3.8)-- (-11.2,4.2);
\draw [line width=1pt,color=fuqqzz] (-11.2,12.2)-- (-9.8,12.2);
\draw [line width=1pt,color=fuqqzz] (-9.8,12.2)-- (-9.8,11.8);
\draw [line width=1pt,color=fuqqzz] (-9.8,11.8)-- (-11.2,11.8);
\draw [line width=1pt,color=fuqqzz] (-11.2,11.8)-- (-11.2,12.2);
\draw [line width=1pt,dotted] (-9.2,12.2)-- (-6.8,12.2);
\draw [line width=1pt,dotted] (-6.8,12.2)-- (-6.8,11.8);
\draw [line width=1pt,dotted] (-6.8,11.8)-- (-9.2,11.8);
\draw [line width=1pt,dotted] (-9.2,11.8)-- (-9.2,12.2);
\begin{scriptsize}
\draw [fill=qqqqff] (-14,5) circle (2.5pt);
\draw [fill=ffffff] (-13,5) circle (2.5pt);
\draw [fill=ffffff] (-12,5) circle (2.5pt);
\draw [fill=qqqqff] (-11,5) circle (2.5pt);
\draw [fill=ffffff] (-10,5) circle (2.5pt);
\draw [fill=ffffff] (-9,5) circle (2.5pt);
\draw [fill=ffffff] (-8,5) circle (2.5pt);
\draw [fill=ffffww] (-7,5) circle (2.5pt);
\draw [fill=ffffff] (-7,4) circle (2.5pt);
\draw [fill=qqffqq] (-8,4) circle (2.5pt);
\draw [fill=ffffff] (-9,4) circle (2.5pt);
\draw [fill=ffffff] (-10,4) circle (2.5pt);
\draw [fill=ffffff] (-11,4) circle (2.5pt);
\draw [fill=ffffff] (-12,4) circle (2.5pt);
\draw [fill=ffffff] (-13,4) circle (2.5pt);
\draw [fill=ffffff] (-14,4) circle (2.5pt);
\draw [fill=ffffff] (-14,3) circle (2.5pt);
\draw [fill=ffffww] (-13,3) circle (2.5pt);
\draw [fill=ffffff] (-12,3) circle (2.5pt);
\draw [fill=ffffff] (-11,3) circle (2.5pt);
\draw [fill=ffffww] (-10,3) circle (2.5pt);
\draw [fill=ffffff] (-9,3) circle (2.5pt);
\draw [fill=ffffff] (-8,3) circle (2.5pt);
\draw [fill=ffffff] (-7,3) circle (2.5pt);
\draw [fill=ffffff] (-7,2) circle (2.5pt);
\draw [fill=qqffqq] (-8,2) circle (2.5pt);
\draw [fill=ffffff] (-9,2) circle (2.5pt);
\draw [fill=ffffff] (-10,2) circle (2.5pt);
\draw [fill=ffffff] (-11,2) circle (2.5pt);
\draw [fill=ffffff] (-12,2) circle (2.5pt);
\draw [fill=ffffff] (-13,2) circle (2.5pt);
\draw [fill=ffffff] (-14,2) circle (2.5pt);
\draw [fill=ffffff] (-14,10) circle (2.5pt);
\draw [fill=ffffff] (-13,10) circle (2.5pt);
\draw [fill=qqffqq] (-12,10) circle (2.5pt);
\draw [fill=ffffff] (-11,10) circle (2.5pt);
\draw [fill=ffffff] (-10,10) circle (2.5pt);
\draw [fill=ffffff] (-9,10) circle (2.5pt);
\draw [fill=ffffff] (-8,10) circle (2.5pt);
\draw [fill=ffffff] (-7,10) circle (2.5pt);
\draw [fill=qqqqff] (-7,11) circle (2.5pt);
\draw [fill=ffffff] (-7,12) circle (2.5pt);
\draw [fill=ffffff] (-7,13) circle (2.5pt);
\draw [fill=ffffff] (-8,13) circle (2.5pt);
\draw [fill=ffffff] (-8,12) circle (2.5pt);
\draw [fill=ffffff] (-8,11) circle (2.5pt);
\draw [fill=ffffff] (-9,11) circle (2.5pt);
\draw [fill=ffffff] (-9,12) circle (2.5pt);
\draw [fill=ffffww] (-9,13) circle (2.5pt);
\draw [fill=ffffff] (-10,13) circle (2.5pt);
\draw [fill=ffffff] (-10,12) circle (2.5pt);
\draw [fill=ffffww] (-10,11) circle (2.5pt);
\draw [fill=ffffff] (-11,11) circle (2.5pt);
\draw [fill=ffffff] (-11,12) circle (2.5pt);
\draw [fill=ffffff] (-11,13) circle (2.5pt);
\draw [fill=ffffff] (-12,13) circle (2.5pt);
\draw [fill=qqffqq] (-12,12) circle (2.5pt);
\draw [fill=ffffff] (-12,11) circle (2.5pt);
\draw [fill=ffffff] (-13,11) circle (2.5pt);
\draw [fill=ffffff] (-13,12) circle (2.5pt);
\draw [fill=ffffff] (-13,13) circle (2.5pt);
\draw [fill=ffffff] (-14,13) circle (2.5pt);
\draw [fill=ffffff] (-14,12) circle (2.5pt);
\draw [fill=ffwwqq] (-14,11) circle (2.5pt);
\draw [fill=cqcqcq] (-14,1) circle (2.5pt);
\draw[color=cqcqcq] (-13.8,1.3) node {$x_1$};
\draw [fill=cqcqcq] (-13,1) circle (2.5pt);
\draw[color=cqcqcq] (-12.8,1.3) node {$x_2$};
\draw [fill=cqcqcq] (-12,1) circle (2.5pt);
\draw[color=cqcqcq] (-11.8,1.3) node {$x_3$};
\draw [fill=cqcqcq] (-11,1) circle (2.5pt);
\draw[color=cqcqcq] (-10.8,1.3) node {$x_4$};
\draw [fill=cqcqcq] (-10,1) circle (2.5pt);
\draw[color=cqcqcq] (-9.8,1.3) node {$x_5$};
\draw [fill=cqcqcq] (-9,1) circle (2.5pt);
\draw[color=cqcqcq] (-8.8,1.3) node {$x_6$};
\draw [fill=cqcqcq] (-8,1) circle (2.5pt);
\draw[color=cqcqcq] (-7.8,1.3) node {$x_7$};
\draw [fill=cqcqcq] (-7,1) circle (2.5pt);
\draw[color=cqcqcq] (-6.8,1.3) node {$x_8$};
\draw [fill=cqcqcq] (-15,2) circle (2.5pt);
\draw[color=cqcqcq] (-14.661108931195733,2.2) node {$y_1$};
\draw [fill=cqcqcq] (-15,3) circle (2.5pt);
\draw[color=cqcqcq] (-14.661108931195733,3.2) node {$y_2$};
\draw [fill=cqcqcq] (-15,4) circle (2.5pt);
\draw[color=cqcqcq] (-14.661108931195733,4.2) node {$y_3$};
\draw [fill=cqcqcq] (-15,5) circle (2.5pt);
\draw[color=cqcqcq] (-14.661108931195733,5.2) node {$y_4$};
\draw [fill=cqcqcq] (-14,9) circle (2.5pt);
\draw[color=cqcqcq] (-13.8,9.3) node {$x_1$};
\draw [fill=cqcqcq] (-13,9) circle (2.5pt);
\draw[color=cqcqcq] (-12.8,9.3) node {$x_2$};
\draw [fill=cqcqcq] (-12,9) circle (2.5pt);
\draw[color=cqcqcq] (-11.8,9.3) node {$x_3$};
\draw [fill=cqcqcq] (-11,9) circle (2.5pt);
\draw[color=cqcqcq] (-10.8,9.3) node {$x_4$};
\draw [fill=cqcqcq] (-10,9) circle (2.5pt);
\draw[color=cqcqcq] (-9.8,9.3) node {$x_5$};
\draw [fill=cqcqcq] (-9,9) circle (2.5pt);
\draw[color=cqcqcq] (-8.8,9.3) node {$x_6$};
\draw [fill=cqcqcq] (-8,9) circle (2.5pt);
\draw[color=cqcqcq] (-7.8,9.3) node {$x_7$};
\draw [fill=cqcqcq] (-7,9) circle (2.5pt);
\draw[color=cqcqcq] (-6.8,9.3) node {$x_8$};
\draw [fill=cqcqcq] (-15,10) circle (2.5pt);
\draw[color=cqcqcq] (-14.661108931195733,10.2) node {$y_1$};
\draw [fill=cqcqcq] (-15,11) circle (2.5pt);
\draw[color=cqcqcq] (-14.661108931195733,11.2) node {$y_2$};
\draw [fill=cqcqcq] (-15,12) circle (2.5pt);
\draw[color=cqcqcq] (-14.661108931195733,12.2) node {$y_3$};
\draw [fill=cqcqcq] (-15,13) circle (2.5pt);
\draw[color=cqcqcq] (-14.661108931195733,13.2) node {$y_4$};
\draw [fill=cqcqcq] (-3,3) circle (2.5pt);
\draw[color=cqcqcq] (-2.6444678620375686,3) node {$z_1$};
\draw [fill=cqcqcq] (-3,11) circle (2.5pt);
\draw[color=cqcqcq] (-2.6444678620375686,11) node {$z_2$};
\draw[color=black] (-2.2956999154393105,7.470884219054397) node {$Z$};
\draw[color=black] (-12.7,7.502590396017874) node {$X$};
\draw[color=black] (-16.78542278774876,10.863445154146481) node {$Y$};
\draw[color=black] (-12.7,-0.4873661987784332) node {$X$};
\draw[color=black] (-16.753716610785283,2.8100762054232176) node {$Y$};
\draw [fill=black] (-14.2,9.8) circle (0.5pt);
\draw [fill=black] (-11.8,9.8) circle (0.5pt);
\draw [fill=black] (-11.8,10.2) circle (0.5pt);
\draw [fill=black] (-14.2,10.2) circle (0.5pt);
\draw [fill=black] (-11.2,10.2) circle (0.5pt);
\draw [fill=black] (-9.8,10.2) circle (0.5pt);
\draw [fill=black] (-9.8,9.8) circle (0.5pt);
\draw [fill=black] (-11.2,9.8) circle (0.5pt);
\draw [fill=black] (-9.2,10.2) circle (0.5pt);
\draw [fill=black] (-6.8,10.2) circle (0.5pt);
\draw [fill=black] (-6.8,9.8) circle (0.5pt);
\draw [fill=black] (-9.2,9.8) circle (0.5pt);
\draw [fill=black] (-11.2,13.2) circle (0.5pt);
\draw [fill=black] (-9.8,13.2) circle (0.5pt);
\draw [fill=black] (-9.8,12.8) circle (0.5pt);
\draw [fill=black] (-11.2,12.8) circle (0.5pt);
\draw [fill=black] (-9.2,13.2) circle (0.5pt);
\draw [fill=black] (-6.8,13.2) circle (0.5pt);
\draw [fill=black] (-6.8,12.790919741852175) circle (0.5pt);
\draw [fill=black] (-9.2,12.790919741852175) circle (0.5pt);
\draw [fill=black] (-14.2,13.2) circle (0.5pt);
\draw [fill=black] (-11.8,13.2) circle (0.5pt);
\draw [fill=black] (-11.8,12.8) circle (0.5pt);
\draw [fill=black] (-14.2,12.8) circle (0.5pt);
\draw [fill=black] (-14.2,5.2) circle (0.5pt);
\draw [fill=black] (-11.8,5.2) circle (0.5pt);
\draw [fill=black] (-11.8,4.8) circle (0.5pt);
\draw [fill=black] (-14.2,4.8) circle (0.5pt);
\draw [fill=black] (-11.2,5.2) circle (0.5pt);
\draw [fill=black] (-9.8,5.2) circle (0.5pt);
\draw [fill=black] (-9.8,4.8) circle (0.5pt);
\draw [fill=black] (-11.2,4.8) circle (0.5pt);
\draw [fill=black] (-9.2,5.2) circle (0.5pt);
\draw [fill=black] (-6.8,5.2) circle (0.5pt);
\draw [fill=black] (-6.78484895589595,4.799937821734619) circle (0.5pt);
\draw [fill=black] (-9.18484895589595,4.799937821734619) circle (0.5pt);
\draw [fill=black] (-11.2,2.2) circle (0.5pt);
\draw [fill=black] (-9.8,2.2) circle (0.5pt);
\draw [fill=black] (-9.8,1.8) circle (0.5pt);
\draw [fill=black] (-11.2,1.8) circle (0.5pt);
\draw [fill=black] (-9.2,2.2) circle (0.5pt);
\draw [fill=black] (-6.8,2.2) circle (0.5pt);
\draw [fill=black] (-6.8,1.8) circle (0.5pt);
\draw [fill=black] (-9.2,1.8) circle (0.5pt);
\draw [fill=black] (-14.2,2.2) circle (0.5pt);
\draw [fill=black] (-11.8,2.2) circle (0.5pt);
\draw [fill=black] (-11.8,1.8) circle (0.5pt);
\draw [fill=black] (-14.2,1.8) circle (0.5pt);
\draw [fill=black] (-14.2,3.2) circle (0.5pt);
\draw [fill=black] (-11.8,3.2) circle (0.5pt);
\draw [fill=black] (-11.8,2.8) circle (0.5pt);
\draw [fill=black] (-14.2,2.8) circle (0.5pt);
\draw [fill=black] (-11.2,3.2) circle (0.5pt);
\draw [fill=black] (-9.8,3.2) circle (0.5pt);
\draw [fill=black] (-9.8,2.8) circle (0.5pt);
\draw [fill=black] (-11.2,2.8) circle (0.5pt);
\draw [fill=black] (-9.2,4.2) circle (0.5pt);
\draw [fill=black] (-6.8,4.2) circle (0.5pt);
\draw [fill=black] (-6.8,3.8) circle (0.5pt);
\draw [fill=black] (-9.2,3.8) circle (0.5pt);
\draw [fill=black] (-14.2,12.2) circle (0.5pt);
\draw [fill=black] (-11.8,12.2) circle (0.5pt);
\draw [fill=black] (-11.8,11.8) circle (0.5pt);
\draw [fill=black] (-14.2,11.8) circle (0.5pt);
\draw [fill=black] (-14.2,11.2) circle (0.5pt);
\draw [fill=black] (-11.8,11.2) circle (0.5pt);
\draw [fill=black] (-11.8,10.8) circle (0.5pt);
\draw [fill=black] (-14.2,10.8) circle (0.5pt);
\draw [fill=black] (-11.2,11.2) circle (0.5pt);
\draw [fill=black] (-9.8,11.2) circle (0.5pt);
\draw [fill=black] (-9.8,10.8) circle (0.5pt);
\draw [fill=black] (-11.2,10.8) circle (0.5pt);
\draw [fill=black] (-9.2,11.2) circle (0.5pt);
\draw [fill=black] (-6.8,11.2) circle (0.5pt);
\draw [fill=black] (-6.8,10.8) circle (0.5pt);
\draw [fill=black] (-9.2,10.8) circle (0.5pt);
\draw [fill=black] (-9.2,3.2) circle (0.5pt);
\draw [fill=black] (-6.8,3.2) circle (0.5pt);
\draw [fill=black] (-6.8,2.8) circle (0.5pt);
\draw [fill=black] (-9.2,2.8) circle (0.5pt);
\draw [fill=black] (-14.2,4.2) circle (0.5pt);
\draw [fill=black] (-11.8,4.2) circle (0.5pt);
\draw [fill=black] (-11.8,3.8) circle (0.5pt);
\draw [fill=black] (-14.2,3.8) circle (0.5pt);
\draw [fill=black] (-11.2,4.2) circle (0.5pt);
\draw [fill=black] (-9.8,4.2) circle (0.5pt);
\draw [fill=black] (-9.8,3.8) circle (0.5pt);
\draw [fill=black] (-11.2,3.8) circle (0.5pt);
\draw [fill=black] (-11.2,12.2) circle (0.5pt);
\draw [fill=black] (-9.8,12.2) circle (0.5pt);
\draw [fill=black] (-9.8,11.8) circle (0.5pt);
\draw [fill=black] (-11.2,11.8) circle (0.5pt);
\draw [fill=black] (-9.2,12.2) circle (0.5pt);
\draw [fill=black] (-6.8,12.2) circle (0.5pt);
\draw [fill=black] (-6.8,11.8) circle (0.5pt);
\draw [fill=black] (-9.2,11.8) circle (0.5pt);
\end{scriptsize}
\end{tikzpicture}
\caption{Cell coloring for the cartesian product $X\square Y\square Z.$}
\label{fig:main fig}
\end{center}
\end{figure}

\end{ex}

Now let us denote by $Y_{i,z}$ (resp.$X_{y,z},$ $Z_{i,y}$ ) the $Y$-fiber (resp. $X$-fiber, $Z$-fiber) in $X\square Y\square Z$ formed by all the cells $\pi_{i}^{y,z}.$ Let $b'_{i,z}$ be the number of blue cells in $Y_{i,z}$ and let $b'$ be the total number of blue cells in $X \Box Y\Box Z.$ We define analogously $c'_{i,z}$ and $c'$ associated with the cells of color $c$ for any $c\in \lbrace \text{green, yellow, orange, red, pink, maroon, white}\rbrace.$ 
In a similar way, for any color $c$ of the eight colors, we define $c'_{y,z}$ for the cells in $X_{y,z},$ and $c'_{i,y}$ for the cells in $Z_{i,y}.$ 

In addition, let $b_{i,z},$ $b_{y,z}$ and $b_{i,y}$ denote the number of blue vertices in $Y_{i,z},$ $X_{y,z}$ and $Z_{i,y}$ respectively and let $b$ be the total number of blue vertices in $X \Box Y\Box Z.$ Similarly, define $g_{i,z},$ $g_{y,z},$ $g_{i,y},$ $y_{i,z},$ $y_{y,z},$ $y_{i,y},$ $o_{i,z},$ $o_{y,z}$ and $o_{i,y}.$

\begin{lem}\label{Main lemma 1} $b'+g'+y'+o'+r'+m'\geq \gamma (X) \gamma (Y) |V(Z)|.$
\end{lem}
\begin{proof}
In the projection $P_Y(Y_{i,z})$ of $Y_{i,z}$ into $Y,$ let each vertex of $Y$ take the color of the projected cell. It would be clear then that any pink or white vertex is dominated by one of the  blue, green, yellow or orange. As such, for any $(i,z)\in [k]\times V(Z),$ we have 
\[
b'_{i,z}+g'_{i,z}+y'_{i,z}+o'_{i,z}+r'_{i,z}+m'_{i,z}\geq \gamma (Y).
\]
Therefore,
\[
b'+g'+y'+o'+r'+m'=\sum_{i=1}^{\gamma (X)}\sum_{z\in V(Z)}b'_{i,z}+g'_{i,z}+y'_{i,z}+o'_{i,z}+r'_{i,z}+m'_{i,z}\geq \gamma(X)\gamma (Y)|V(Z)|.
\]
\end{proof}

In a similar way, one can prove the following lemma. We omit its proof since it will bot be used lately.

\begin{lem}\label{Main lemma 2} $b'+g'+y'+o'+r'+p'\geq \gamma (X) \gamma (Z) |V(Y)|.$
\end{lem}

\begin{lem}\label{Main lemma 3} $b'+r'\leq b+g+y+o.$
\end{lem}
\begin{proof}
For any $(y,z)\in V(Y)\times V(Z),$ in $X_{y,z},$ the blue, green, yellow and orange vertices dominate all the vertices in the blue and red cells. As such if we add to them the vertices $(u_i,y,z)$ of each green, yellow, orange, pink, maroon and white, we dominate all the vertices in $X_{y,z}.$ This shows that for any $(y,z)\in V(Y)\times V(Z),$
\[
b_{y,z}+g_{y,z}+y_{y,z}+o_{y,z}+g'_{y,z}+y'_{y,z}+o'_{y,z}+p'_{y,z}+m'_{y,z}+w'_{y,z}\geq \gamma (X)
\]
Therefore, by summing over all $(y,z)\in V(Y)\times V(Z),$ we get
\begin{equation}\label{Eq1}
b+g+y+o+g'+y'+o'+p'+m'+w'\geq \gamma (X)|V(Y)||V(Z)|.
\end{equation}
But since the total number of all cells is $\gamma (X)|V(Y)||V(Z)|,$ we have
\begin{equation}\label{Eq2}
b'+g'+y'+o'+r'+p'+m'+w'= \gamma (X)|V(Y)||V(Z)|.
\end{equation}
Using Equations $(\ref{Eq1})$ and $(\ref{Eq2})$ we get:
\begin{equation*}
b+g+y+o+(\gamma (X)|V(Y)||V(Z)|-b'-r')\geq \gamma (X)|V(Y)||V(Z)|.
\end{equation*}
The result follows.
\end{proof}

The reader is invited to check the results of the previous lemmas for Example \ref{Main example}.

\begin{theoreme}(Clark and Suen, \cite{2}) For any pair of graphs $X$ and $Y$ we have 
\[
\gamma (X\square Y)\geq \frac{1}{2}\gamma (X)\gamma (Y).
\]
\end{theoreme}
\begin{proof}
In the case where $Z$ is the trivial graph with only one vertex, there will be no yellow, orange, maroon and white cells. As such, by Lemma \ref{Main lemma 1} and Lemma \ref{Main lemma 3}, we have:
\[ b'+g'+r'\geq \gamma (X) \gamma (Y) \]
and 
\[ b+g\geq b'+r'. \]
Combining these two inequalities, we get:
\[b+g+g'\geq \gamma (X) \gamma (Y).\]
The result follows from the fact that $\gamma(X\square Y)=b+g.$
\end{proof}

\begin{theoreme}\label{Main Theorem} For any pair of graphs $X$ and $Y$ we have 
\[
\gamma (X\square Y\square P_2)\geq \frac{2}{3}\gamma (X)\gamma (Y).
\]
\end{theoreme}
\begin{proof}
In the case where $Z$ is the path graph $P_2,$ by Lemma \ref{Main lemma 1} and Lemma \ref{Main lemma 3}, we have:
\begin{equation}\label{Eq3}
b+g+y+o+g'+y'+o'+m'\geq 2\gamma (X) \gamma (Y) .
\end{equation}
Let us call $\pi_i^{y,z_2}$ (resp. $\pi_i^{y,z_1}$) the complement of the cell $\pi_i^{y,z_1}$ (resp. $\pi_i^{y,z_2}$). Now remark that if a cell is colored maroon then its complement should be colored blue or green. This is because a maroon cell can't be only $X$-dominated, as such the intersection of $D$ and its complement is not empty. This implies that the complement of a maroon cell is either blue, green, yellow or orange. But it can't be yellow or orange since yellow and orange cells have $Z$-dominated vertices. In a similar way, the complement of a white cell is blue or green. Therefore, we have:
\begin{equation}
 b'+g'= m'+w'.
\end{equation}
In particular,
\begin{equation}\label{Eq4}
 b'+g'\geq m'.
\end{equation}
Combining (\ref{Eq3}) and (\ref{Eq4}), we get:
\[ b+g+y+o+g'+y'+o'+b'+g' \geq 2\gamma (X) \gamma (Y).\]
The result follows using the fact that $\gamma (X\square Y\square P_2)=|D|=b+g+y+o\geq b'+g'+y'+o'.$
\end{proof}

In what remains, we fix $Z$ to be the path graph $P_n$ on the $n$ vertices $\lbrace z_1,\cdots , z_n\rbrace.$ We state the following observations for the cells in any $Z$-fiber $Z_{i,y}$ in $X\square Y\square P_n,$
\begin{enumerate}
\item[$O_1:$] For any $1\leq l\leq n-2,$ the three consecutive cells $\pi_{i}^{y,z_l},$ $\pi_{i}^{y,z_{l+1}}$ and $\pi_{i}^{y,z_{l+2}}$ can't be all colored maroon.
\item[$O_2:$] The first two consecutive cells $\pi_{i}^{y,z_1}$ and $\pi_{i}^{y,z_{2}}$ can't be both maroon. Same goes for the last two consecutive cells $\pi_{i}^{y,z_{n-1}}$ and $\pi_{i}^{y,z_{n}}$
\item[$O_3:$] For any $2\leq l\leq n-2,$ if the two consecutive cells $\pi_{i}^{y,z_l}$ and $\pi_{i}^{y,z_{l+1}}$ are colored maroon then each of the cells $\pi_{i}^{y,z_{l-1}}$ and $\pi_{i}^{y,z_{l+2}}$ is blue, green, yellow or orange.
\item[$O_4:$] If the cell $\pi_{i}^{y,z_l}$ is maroon, then at least one of the two cells $\pi_{i}^{y,z_{l-1}}$ and $\pi_{i}^{y,z_{l+1}}$ is blue, green, yellow or orange.
\item[$O_5:$] If the cell $\pi_{i}^{y,z_l}$ is blue or green, then the two cells $\pi_{i}^{y,z_{l-1}}$ and $\pi_{i}^{y,z_{l-1}}$ can only be maroon or white.
\item[$O_6:$] If the cell $\pi_{i}^{y,z_l}$ is colored yellow or orange, then at least one of the two cells $\pi_{i}^{y,z_{l-1}}$ and $\pi_{i}^{y,z_{l-1}}$ should be yellow or orange and none of them can be colored red or pink.
\end{enumerate}

\begin{lem}\label{Main lem} For any non-negative integer $k,$ we have
 $$m'\leq 2\gamma (X\square Y\square P_n),~~~~  \text{if $n=3k$, with $k\geq 1,$}$$
 $$m'\leq \frac{2k}{k+1}\gamma (X\square Y\square P_n),~~~~  \text{if $n=3k+1$,}$$
 $$m'\leq \frac{2k+1}{k+1}\gamma (X\square Y\square P_n),~~~~  \text{if $n=3k+2$,}$$
\end{lem}
\begin{proof}
When $Z=P_n,$ let us compare the number of maroon cells with the number of blue, green, yellow and orange cells in the $Z$-fiber $Z_{i,y}$ for any $(i,y)\in [k]\times V(Y).$ By the observations listed above, when $n=3k$ for $k\in \mathbb{N},$ the maximum number of maroon cells in $Z_{i,y}$ will be $2k$ which is  reached in the situation represented graphically in Figure \ref{fig:max maroon}.

\definecolor{ududff}{rgb}{0.30196078431372547,0.30196078431372547,1}
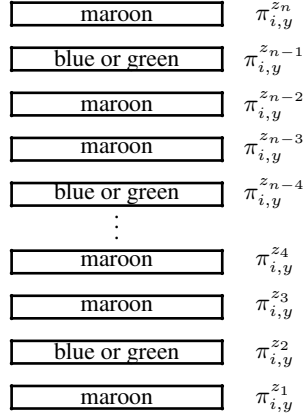
\begin{figure}[htbp]
\begin{center}
\begin{tikzpicture}[line cap=round,line join=round,>=triangle 45,x=.7cm,y=.3cm]
%\clip(-15.033982141286769,-4.277303549979223) rectangle (-0.7081603059516146,4.555882559252097);
\draw [line width=1pt] (-11,10)-- (-7,10);
\draw [line width=1pt] (-7,10)-- (-7,9);
\draw [line width=1pt] (-7,9)-- (-11,9);
\draw [line width=1pt] (-11,9)-- (-11,10);
\draw [line width=1pt] (-11,8)-- (-7,8);
\draw [line width=1pt] (-7,8)-- (-7,7);
\draw [line width=1pt] (-7,7)-- (-11,7);
\draw [line width=1pt] (-11,7)-- (-11,8);
\draw [line width=1pt] (-11,6)-- (-7,6);
\draw [line width=1pt] (-7,6)-- (-7,5);
\draw [line width=1pt] (-7,5)-- (-11,5);
\draw [line width=1pt] (-11,5)-- (-11,6);
\draw [line width=1pt] (-11,4)-- (-7,4);
\draw [line width=1pt] (-7,4)-- (-7,3);
\draw [line width=1pt] (-7,3)-- (-11,3);
\draw [line width=1pt] (-11,3)-- (-11,4);
\draw [line width=1pt] (-11,2)-- (-7,2);
\draw [line width=1pt] (-7,2)-- (-7,1);
\draw [line width=1pt] (-7,1)-- (-11,1);
\draw [line width=1pt] (-11,1)-- (-11,2);
\draw [line width=1pt] (-11,-1)-- (-7,-1);
\draw [line width=1pt] (-7,-1)-- (-7,-2);
\draw [line width=1pt] (-7,-2)-- (-11,-2);
\draw [line width=1pt] (-11,-2)-- (-11,-1);
\draw [line width=1pt] (-11,-3)-- (-7,-3);
\draw [line width=1pt] (-7,-3)-- (-7,-4);
\draw [line width=1pt] (-7,-4)-- (-11,-4);
\draw [line width=1pt] (-11,-4)-- (-11,-3);
\draw [line width=1pt] (-11,-5)-- (-7,-5);
\draw [line width=1pt] (-7,-5)-- (-7,-6);
\draw [line width=1pt] (-7,-6)-- (-11,-6);
\draw [line width=1pt] (-11,-6)-- (-11,-5);
\draw [line width=1pt] (-11,-7)-- (-7,-7);
\draw [line width=1pt] (-7,-7)-- (-7,-8);
\draw [line width=1pt] (-7,-8)-- (-11,-8);
\draw [line width=1pt] (-11,-8)-- (-11,-7);
\begin{scriptsize}
\draw [fill=black] (-11,10) circle (0.5pt);
\draw [fill=black] (-7,10) circle (0.5pt);
\draw [fill=black] (-7,9) circle (0.5pt);
\draw [fill=black] (-11,9) circle (0.5pt);
\draw[color=black] (-9,9.430195257709743) node {maroon};
\draw[color=black] (-6,9.430195257709743) node {$\pi_{i,y}^{z_{n}}$};
\draw [fill=black] (-11,8) circle (0.5pt);
\draw [fill=black] (-7,8) circle (0.5pt);
\draw [fill=black] (-7,7) circle (0.5pt);
\draw [fill=black] (-11,7) circle (0.5pt);
\draw[color=black] (-9,7.438713298537591) node {blue or green};
\draw[color=black] (-6,7.438713298537591) node {$\pi_{i,y}^{z_{n-1}}$};
\draw [fill=black] (-11,6) circle (0.5pt);
\draw [fill=black] (-7,6) circle (0.5pt);
\draw [fill=black] (-7,5) circle (0.5pt);
\draw [fill=black] (-11,5) circle (0.5pt);
\draw[color=black] (-9,5.431171000985018) node {maroon};
\draw[color=black] (-6,5.431171000985018) node {$\pi_{i,y}^{z_{n-2}}$};
\draw [fill=black] (-11,4) circle (0.5pt);
\draw [fill=black] (-7,4) circle (0.5pt);
\draw [fill=black] (-7,3) circle (0.5pt);
\draw [fill=black] (-11,3) circle (0.5pt);
\draw[color=black] (-9,3.5521114104758107) node {maroon};
\draw[color=black] (-6,3.5521114104758107) node {$\pi_{i,y}^{z_{n-3}}$};
\draw [fill=black] (-11,2) circle (0.5pt);
\draw [fill=black] (-7,2) circle (0.5pt);
\draw [fill=black] (-7,1) circle (0.5pt);
\draw [fill=black] (-11,1) circle (0.5pt);
\draw[color=black] (-9,1.5445691129232377) node {blue or green};
\draw[color=black] (-6,1.5445691129232377) node {$\pi_{i,y}^{z_{n-4}}$};
\draw [fill=black] (-11,-1) circle (0.5pt);
\draw [fill=black] (-7,-1) circle (0.5pt);
\draw [fill=black] (-7,-2) circle (0.5pt);
\draw [fill=black] (-11,-2) circle (0.5pt);
\draw[color=black] (-9,-1.458714164215411) node {maroon};
\draw[color=black] (-6,-1.458714164215411) node {$\pi_{i,y}^{z_4}$};
\draw [fill=black] (-11,-3) circle (0.5pt);
\draw [fill=black] (-7,-3) circle (0.5pt);
\draw [fill=black] (-7,-4) circle (0.5pt);
\draw [fill=black] (-11,-4) circle (0.5pt);
\draw[color=black] (-9,-3.4180754466267214) node {maroon};
\draw[color=black] (-6,-3.4180754466267214) node {$\pi_{i,y}^{z_3}$};
\draw [fill=black] (-11,-5) circle (0.5pt);
\draw [fill=black] (-7,-5) circle (0.5pt);
\draw [fill=black] (-7,-6) circle (0.5pt);
\draw [fill=black] (-11,-6) circle (0.5pt);
\draw[color=black] (-9,-5.554100451222659) node {blue or green};
\draw[color=black] (-6,-5.554100451222659) node {$\pi_{i,y}^{z_2}$};
\draw [fill=black] (-11,-7) circle (0.5pt);
\draw [fill=black] (-7,-7) circle (0.5pt);
\draw [fill=black] (-7,-8) circle (0.5pt);
\draw [fill=black] (-11,-8) circle (0.5pt);
\draw[color=black] (-9,-7.545582410394811) node {maroon};
\draw[color=black] (-6,-7.545582410394811) node {$\pi_{i,y}^{z_1}$};
%\draw [fill=ududff] (-9,0) circle (2.5pt);
\draw[color=black] (-9,0.38822474953295605) node {$\vdots$};
\end{scriptsize}
\end{tikzpicture}
\caption{The situation when the maximum number of maroon cells is reached in a $Z$-fiber $Z_{i,y}$ in $X\square Y\square P_n.$}
\label{fig:max maroon}
\end{center}
\end{figure}

In this situation, the number of blue and green cells in $Z_{i,y}$ will be $k.$ As such, taking into account the observation $O_6,$ for any $(i,y)\in [k]\times V(Y),$ we have the following inequality:
\begin{equation}
2(b'_{i,y}+g'_{i,y})+y'_{i,y}+o'_{i,y}\geq m'_{i,y}
\end{equation} 
By summing over all pairs $(i,y)\in [k]\times V(Y),$ and using the fact that $\gamma (X\square Y\square P_n)=|D|=b+g+y+o\geq b'+g'+y'+o',$ we get:
$$2\gamma (X\square Y\square P_n)\geq m',~~~~  \text{if $n=3k$ with $k\in \mathbb{N}.$}$$
To obtain the remaining two inequalities, observe that when $n=3k+1,$ for an integer $k\geq 0,$ the maximum number of maroon cells in $Z_{i,y}$ will be $2k$ which is reached in the situation where the number of blue and green cells in $Z_{i,y}$ is $k+1.$ In a similar way, when $n=3k+2,$ for an integer $k\geq 0,$ the maximum number of maroon cells in $Z_{i,y}$ will be $2k+1$ which is reached in the situation where the number of blue and green cells in $Z_{i,y}$ is $k+1.$

\end{proof}

\begin{ex}\label{Example 2} Consider the three graphs $X, Y$ and $Z$ with vertex sets \newline $V(X)=\lbrace x_1,x_2,x_3,x_4,x_5,x_6,x_7,x_8,x_9\rbrace,$ $V(Y)=\lbrace y_1,y_2,y_3,y_4,y_5\rbrace,$ $V(Z)=\lbrace z_1,z_2,z_3\rbrace$ respectively and edge sets defined as follows 
$$E(X)=\lbrace x_1x_2,x_2x_3,x_2x_6,x_3x_4,x_4x_5,x_5x_6,x_5x_8,x_6x_7,x_7x_8,x_8x_9\rbrace,$$ 
$$E(Y)=\lbrace y_1y_2,y_1y_4,y_2y_3,y_3y_4,y_4y_5\rbrace,$$ 
and 
$$E(Z)=\lbrace z_1z_2,z_2z_3\rbrace.$$ 
It would be easy to check that the domination number of $X$ is $3,$ of $Y$ is $2$ and that of $Z$ is $1.$ Consider the dominating set $\lbrace x_2,x_5,x_8\rbrace$ of $X.$ Let $\lbrace \pi_{1}, \pi_{2},\pi_{3}\rbrace$ be the partition of $V(X)$ where:
$$\pi_{1}=\lbrace x_1,x_2,x_3,x_6\rbrace,$$
$$\pi_{2}=\lbrace x_4,x_5\rbrace$$
and
$$\pi_{3}=\lbrace x_7,x_8,x_9\rbrace.$$

We used SageMath \cite{7} to verify that the domination number of $X\square Y\square Z$ is $25$ and that a minimal dominating set is given by 
$$D=\lbrace (x_1, y_4, z_1),
 (x_2, y_2,z_1),
 (x_3, y_4,z_1),
 (x_4, y_4,z_1),
 (x_6, y_5,z_1),
 (x_8, y_1,z_1),
 (x_8, y_2,z_1),$$
 $$
 (x_8, y_3,z_1),
 (x_9, y_5,z_1),
 (x_1, y_4,z_2),
 (x_4, y_2,z_2),
 (x_4, y_5,z_2),
(x_6, y_1,z_2),
 (x_6, y_3,z_2),$$
 $$ 
 (x_7, y_4,z_2),
 (x_9, y_2,z_2),
 (x_1, y_2,z_3),
 (x_2, y_5,z_3),
 (x_3, y_1,z_3),
(x_3, y_3,z_3),$$
$$
 (x_5, y_4,z_3), 
 (x_7, y_2,z_3),
 (x_8, y_2,z_3),
 (x_8, y_5,z_3),
 (x_9, y_4,z_3)\rbrace.$$

In Figure \ref{fig:main fig}, we show the coloring of the cells with eight colors as well as the coloring of the vertices of $D$ with the four colors blue, green, yellow and orange. Similar to Example \ref{Main example}, we use dashed lines for white colored cells. The cells in the bottom correspond to the layer $(X\square Y)^{z_1},$ those in the middle correspond to layer $(X\square Y)^{z_2},$ while on the top figures the cells of the layer $(X\square Y)^{z_3}.$ The edges in the cartesian product $X\square Y\square Z$ are omitted to make the picture easily readable. In this example, the maximum number of maroon cells is reached in $Z_{2,y_2}.$ The reader is invited to check that this agrees with the graphic representation in Figure \ref{fig:max maroon}.
\definecolor{ffxfqq}{rgb}{1,0.4980392156862745,0}
\definecolor{ccwwqq}{rgb}{0.8,0.4,0}
\definecolor{ffzzcc}{rgb}{1,0.6,0.8}
\definecolor{cqcqcq}{rgb}{0.7529411764705882,0.7529411764705882,0.7529411764705882}
\definecolor{ffffww}{rgb}{1,1,0.4}
\definecolor{qqffqq}{rgb}{0,1,0}
\definecolor{qqqqff}{rgb}{0,0,1}
\definecolor{ffffff}{rgb}{1,1,1}
\definecolor{ffwwqq}{rgb}{1,0.4,0}
\begin{figure}[htbp]
\begin{center}
\begin{tikzpicture}[line cap=round,line join=round,>=triangle 45,x=.8cm,y=.85cm]
%\clip(-34.05336649426413,-3.7377802312907904) rectangle (10.961114586593732,24.332474410968203);
\draw [line width=1pt] (-14,0)-- (-13,0);
\draw [line width=1pt] (-13,0)-- (-12,0);
\draw [line width=1pt] (-12,0)-- (-11,0);
\draw [line width=1pt] (-11,0)-- (-10,0);
\draw [line width=1pt] (-10,0)-- (-9,0);
\draw [line width=1pt] (-9,0)-- (-8,0);
\draw [line width=1pt] (-8,0)-- (-7,0);
\draw [line width=1pt] (-15,1)-- (-15,2);
\draw [line width=1pt] (-15,2)-- (-15,3);
\draw [line width=1pt] (-15,3)-- (-15,4);
\draw [line width=1pt] (-14,9)-- (-13,9);
\draw [line width=1pt] (-13,9)-- (-12,9);
\draw [line width=1pt] (-12,9)-- (-11,9);
\draw [line width=1pt] (-11,9)-- (-10,9);
\draw [line width=1pt] (-10,9)-- (-9,9);
\draw [line width=1pt] (-9,9)-- (-8,9);
\draw [line width=1pt] (-8,9)-- (-7,9);
\draw [line width=1pt] (-15,10)-- (-15,11);
\draw [line width=1pt] (-15,11)-- (-15,12);
\draw [line width=1pt] (-15,12)-- (-15,13);
\draw [line width=1pt] (-2.994015505438793,1.997063168701213)-- (-3,12);
\draw [shift={(-8.5,9)},line width=1pt]  plot[domain=3.141592653589793:6.283185307179586,variable=\t]({1*1.5*cos(\t r)+0*1.5*sin(\t r)},{0*1.5*cos(\t r)+1*1.5*sin(\t r)});
\draw [shift={(-15,11.5)},line width=1pt]  plot[domain=1.5707963267948966:4.71238898038469,variable=\t]({1*1.5*cos(\t r)+0*1.5*sin(\t r)},{0*1.5*cos(\t r)+1*1.5*sin(\t r)});
\draw [shift={(-8.5,0)},line width=1pt]  plot[domain=3.141592653589793:6.283185307179586,variable=\t]({1*1.5*cos(\t r)+0*1.5*sin(\t r)},{0*1.5*cos(\t r)+1*1.5*sin(\t r)});
\draw [shift={(-15,2.5)},line width=1pt]  plot[domain=1.5707963267948966:4.71238898038469,variable=\t]({1*1.5*cos(\t r)+0*1.5*sin(\t r)},{0*1.5*cos(\t r)+1*1.5*sin(\t r)});
\draw [line width=1pt,dotted] (-14.2,2.8)-- (-8.8,2.8);
\draw [line width=1pt,dotted] (-8.8,2.8)-- (-8.8,3.2);
\draw [line width=1pt,dotted] (-8.8,3.2)-- (-14.2,3.2);
\draw [line width=1pt,dotted] (-14.2,3.2)-- (-14.2,2.8);
\draw [line width=1pt,color=qqqqff] (-11.1,4.1)-- (-9.9,4.1);
\draw [line width=1pt,color=qqqqff] (-9.9,4.1)-- (-9.9,3.9);
\draw [line width=1pt,color=qqqqff] (-9.9,3.9)-- (-11.1,3.9);
\draw [line width=1pt,color=qqqqff] (-11.1,3.9)-- (-11.1,4.1);
\draw [line width=1pt,dotted] (-5.8,10.2)-- (-5.8,9.8);
\draw [line width=1pt,dotted] (-5.8,9.8)-- (-8.2,9.8);
\draw [line width=1pt,dotted] (-8.2,9.8)-- (-8.2,10.2);
\draw [line width=1pt,color=qqqqff] (-11.1,14.1)-- (-9.9,14.1);
\draw [line width=1pt,color=qqqqff] (-9.9,13.9)-- (-11.1,13.9);
\draw [line width=1pt,color=qqqqff] (-11.1,13.9)-- (-11.1,14.1);
\draw [line width=1pt,color=ffzzcc] (-11.1,3.1)-- (-9.9,3.1);
\draw [line width=1pt,color=ffzzcc] (-9.9,3.1)-- (-9.9,2.9);
\draw [line width=1pt,color=ffzzcc] (-9.9,2.9)-- (-11.1,2.9);
\draw [line width=1pt,color=ffzzcc] (-11.1,2.9)-- (-11.1,3.1);
\draw [line width=1pt,color=ffwwqq] (-14.2,4.2)-- (-8.8,4.2);
\draw [line width=1pt,color=ffwwqq] (-8.8,4.2)-- (-8.8,3.8);
\draw [line width=1pt,color=ffwwqq] (-8.8,3.8)-- (-14.2,3.8);
\draw [line width=1pt,color=ffwwqq] (-14.2,3.8)-- (-14.2,4.2);
\draw [line width=1pt,color=qqqqff] (-14.2,2.2)-- (-8.8,2.2);
\draw [line width=1pt,color=qqqqff] (-8.8,2.2)-- (-8.8,1.8);
\draw [line width=1pt,color=qqqqff] (-8.8,1.8)-- (-14.2,1.8);
\draw [line width=1pt,color=qqqqff] (-14.2,1.8)-- (-14.2,2.2);
\draw [line width=1pt,color=ccwwqq] (-11.1,2.1)-- (-9.9,2.1);
\draw [line width=1pt,color=ccwwqq] (-9.9,1.9)-- (-11.1,1.9);
\draw [line width=1pt,color=ccwwqq] (-11.1,1.9)-- (-11.1,2.1);
\draw [line width=1pt,color=ffwwqq] (-8.2,2.2)-- (-5.8,2.2);
\draw [line width=1pt,color=ffwwqq] (-5.8,2.2)-- (-5.8,1.8);
\draw [line width=1pt,color=ffwwqq] (-8.2,1.8)-- (-8.2,2.2);
\draw [line width=1pt,color=ffwwqq] (-8.2,1.8)-- (-5.8,1.8);
\draw [line width=1pt,color=ffzzcc] (-11.1,1.1)-- (-9.9,1.1);
\draw [line width=1pt,color=ffzzcc] (-9.9,1.1)-- (-9.9,0.9);
\draw [line width=1pt,color=ffzzcc] (-9.9,0.9)-- (-11.1,0.9);
\draw [line width=1pt,color=ffzzcc] (-11.1,0.9)-- (-11.1,1.1);
\draw [line width=1pt,color=qqffqq] (-5.8,1.2)-- (-5.8,0.8);
\draw [line width=1pt,color=qqffqq] (-5.8,0.8)-- (-8.2,0.8);
\draw [line width=1pt,color=qqffqq] (-8.2,0.8)-- (-8.2,1.2);
\draw [line width=1pt,dotted] (-14.2,1.2)-- (-8.8,1.2);
\draw [line width=1pt,dotted] (-8.8,1.2)-- (-8.8,0.8);
\draw [line width=1pt,dotted] (-8.8,0.8)-- (-14.2,0.8);
\draw [line width=1pt,dotted] (-14.2,0.8)-- (-14.2,1.2);
\draw [line width=1pt] (-15,5)-- (-15,4);
\draw [line width=1pt] (-15,14)-- (-15,13);
\draw [line width=1pt] (-14,18)-- (-13,18);
\draw [line width=1pt] (-13,18)-- (-12,18);
\draw [line width=1pt] (-12,18)-- (-11,18);
\draw [line width=1pt] (-11,18)-- (-10,18);
\draw [line width=1pt] (-10,18)-- (-9,18);
\draw [line width=1pt] (-9,18)-- (-8,18);
\draw [line width=1pt] (-8,18)-- (-7,18);
\draw [shift={(-8.5,18)},line width=1pt]  plot[domain=3.141592653589793:6.283185307179586,variable=\t]({1*1.5*cos(\t r)+0*1.5*sin(\t r)},{0*1.5*cos(\t r)+1*1.5*sin(\t r)});
\draw [line width=1pt] (-15,19)-- (-15,20);
\draw [line width=1pt] (-15,20)-- (-15,21);
\draw [line width=1pt] (-15,21)-- (-15,22);
\draw [line width=1pt] (-15,22)-- (-15,23);
\draw [shift={(-15,20.5)},line width=1pt]  plot[domain=1.5707963267948966:4.71238898038469,variable=\t]({1*1.5*cos(\t r)+0*1.5*sin(\t r)},{0*1.5*cos(\t r)+1*1.5*sin(\t r)});
\draw [line width=1pt] (-3,21)-- (-3,12);
\draw [line width=1pt,dotted] (-5.8,13.8)-- (-5.8,14.2);
\draw [line width=1pt,dotted] (-5.8,14.2)-- (-8.2,14.2);
\draw [line width=1pt,dotted] (-8.2,14.2)-- (-8.2,13.8);
\draw [line width=1pt,color=qqffqq] (-8.2,23.2)-- (-5.8,23.2);
\draw [line width=1pt,color=qqffqq] (-5.8,23.2)-- (-5.8,22.8);
\draw [line width=1pt,color=qqffqq] (-8.2,22.8)-- (-8.2,23.2);
\draw [line width=1pt,color=ccwwqq] (-11.1,19.9)-- (-9.9,19.9);
\draw [line width=1pt,color=ccwwqq] (-9.9,20.1)-- (-11.1,20.1);
\draw [line width=1pt,color=ccwwqq] (-11.1,20.1)-- (-11.1,19.9);
\draw [line width=1pt,color=ffzzcc] (-8.2,20.8)-- (-5.8,20.8);
\draw [line width=1pt,color=ffzzcc] (-5.8,20.8)-- (-5.8,21.2);
\draw [line width=1pt,color=ffzzcc] (-5.8,21.2)-- (-8.2,21.2);
\draw [line width=1pt,color=ffzzcc] (-8.2,21.2)-- (-8.2,20.8);
\draw [line width=1pt,color=ffzzcc] (-8.2,18.8)-- (-5.8,18.8);
\draw [line width=1pt,color=ffzzcc] (-5.8,18.8)-- (-5.8,19.2);
\draw [line width=1pt,color=ffzzcc] (-8.2,19.2)-- (-8.2,18.8);
\draw [line width=1pt,color=qqffqq] (-14.2,20.2)-- (-8.8,20.2);
\draw [line width=1pt,color=qqffqq] (-8.8,20.2)-- (-8.8,19.8);
\draw [line width=1pt,color=qqffqq] (-8.8,19.8)-- (-14.2,19.8);
\draw [line width=1pt,color=qqffqq] (-14.2,19.8)-- (-14.2,20.2);
\draw [line width=1pt,color=ffxfqq] (-8.8,19.2)-- (-14.2,19.2);
\draw [line width=1pt,color=ffxfqq] (-14.2,19.2)-- (-14.2,18.8);
\draw [line width=1pt,color=ffxfqq] (-14.2,18.8)-- (-8.8,18.8);
\draw [line width=1pt,color=ffxfqq] (-8.8,18.8)-- (-8.8,19.2);
\draw [line width=1pt,color=ffzzcc] (-11.1,19.1)-- (-9.9,19.1);
\draw [line width=1pt,color=ffzzcc] (-9.9,19.1)-- (-9.9,18.9);
\draw [line width=1pt,color=ffzzcc] (-9.9,18.9)-- (-11.1,18.9);
\draw [line width=1pt,color=ffzzcc] (-11.1,18.9)-- (-11.1,19.1);
\draw [shift={(-11,0)},line width=1pt]  plot[domain=3.141592653589793:6.283185307179586,variable=\t]({1*2*cos(\t r)+0*2*sin(\t r)},{0*2*cos(\t r)+1*2*sin(\t r)});
\draw [shift={(-11,9)},line width=1pt]  plot[domain=3.141592653589793:6.283185307179586,variable=\t]({1*2*cos(\t r)+0*2*sin(\t r)},{0*2*cos(\t r)+1*2*sin(\t r)});
\draw [shift={(-11,18)},line width=1pt]  plot[domain=3.141592653589793:6.283185307179586,variable=\t]({1*2*cos(\t r)+0*2*sin(\t r)},{0*2*cos(\t r)+1*2*sin(\t r)});
\draw [line width=1pt,color=qqffqq] (-8.2,2.8)-- (-8.2,3.2);
\draw [line width=1pt,color=qqffqq] (-8.2,3.2)-- (-5.8,3.2);
\draw [line width=1pt,color=qqffqq] (-5.8,3.2)-- (-5.8,2.8);
\draw [line width=1pt,color=ffffww] (-8.2,12.8)-- (-8.2,13.2);
\draw [line width=1pt,color=ccwwqq] (-9.9,2.1)-- (-9.9,1.9);
\draw [line width=1pt,color=qqffqq] (-8.2,1.2)-- (-5.8,1.2);
\draw [line width=1pt,color=qqffqq] (-8.2,2.8)-- (-5.8,2.8);
\draw [line width=1pt,color=qqqqff] (-8.2,5.2)-- (-5.8,5.2);
\draw [line width=1pt,color=qqqqff] (-5.8,5.2)-- (-5.8,4.8);
\draw [line width=1pt,color=qqqqff] (-5.8,4.8)-- (-8.2,4.8);
\draw [line width=1pt,color=qqqqff] (-8.2,4.8)-- (-8.2,5.2);
\draw [line width=1pt,dotted] (-8.2,4.2)-- (-5.8,4.2);
\draw [line width=1pt,dotted] (-5.8,4.2)-- (-5.8,3.8);
\draw [line width=1pt,dotted] (-5.8,3.8)-- (-8.2,3.8);
\draw [line width=1pt,dotted] (-8.2,3.8)-- (-8.2,4.2);
\draw [line width=1pt,dotted] (-11.1,5.1)-- (-9.9,5.1);
\draw [line width=1pt,dotted] (-9.9,5.1)-- (-9.9,4.9);
\draw [line width=1pt,dotted] (-9.9,4.9)-- (-11.1,4.9);
\draw [line width=1pt,dotted] (-11.1,4.9)-- (-11.1,5.1);
\draw [line width=1pt,color=qqffqq] (-14.2,5.2)-- (-8.8,5.2);
\draw [line width=1pt,color=qqffqq] (-8.8,5.2)-- (-8.8,4.8);
\draw [line width=1pt,color=qqffqq] (-8.8,4.8)-- (-14.2,4.8);
\draw [line width=1pt,color=qqffqq] (-14.2,4.8)-- (-14.2,5.2);
\draw [line width=1pt,dotted] (-8.2,13.8)-- (-5.8,13.8);
\draw [line width=1pt,color=qqqqff] (-9.9,14.1)-- (-9.9,13.9);
\draw [line width=1pt,dotted] (-8.2,10.2)-- (-5.8,10.2);
\draw [line width=1pt,color=ffffww] (-8.2,13.2)-- (-5.8,13.2);
\draw [line width=1pt,color=ffffww] (-5.8,13.2)-- (-5.8,12.8);
\draw [line width=1pt,color=ffffww] (-5.8,12.8)-- (-8.2,12.8);
\draw [line width=1pt,dotted] (-14.2,14.2)-- (-8.8,14.2);
\draw [line width=1pt,dotted] (-8.8,14.2)-- (-8.8,13.8);
\draw [line width=1pt,dotted] (-8.8,13.8)-- (-14.2,13.8);
\draw [line width=1pt,dotted] (-14.2,13.8)-- (-14.2,14.2);
\draw [line width=1pt,color=ffwwqq] (-14.2,12.2)-- (-8.8,12.2);
\draw [line width=1pt,color=ffwwqq] (-8.8,12.2)-- (-8.8,11.8);
\draw [line width=1pt,color=ffwwqq] (-8.8,11.8)-- (-14.2,11.8);
\draw [line width=1pt,color=ffwwqq] (-14.2,11.8)-- (-14.2,12.2);
\draw [line width=1pt,dotted] (-14.2,11.2)-- (-8.8,11.2);
\draw [line width=1pt,dotted] (-8.8,11.2)-- (-8.8,10.8);
\draw [line width=1pt,dotted] (-8.8,10.8)-- (-14.2,10.8);
\draw [line width=1pt,dotted] (-14.2,10.8)-- (-14.2,11.2);
\draw [line width=1pt,color=qqffqq] (-14.2,10.2)-- (-8.8,10.2);
\draw [line width=1pt,color=qqffqq] (-8.8,10.2)-- (-8.8,9.8);
\draw [line width=1pt,color=qqffqq] (-8.8,9.8)-- (-14.2,9.8);
\draw [line width=1pt,color=qqffqq] (-14.2,9.8)-- (-14.2,10.2);
\draw [line width=1pt,color=ffzzcc] (-11.1,10.1)-- (-9.9,10.1);
\draw [line width=1pt,color=ffzzcc] (-9.9,10.1)-- (-9.9,9.9);
\draw [line width=1pt,color=ffzzcc] (-9.9,9.9)-- (-11.1,9.9);
\draw [line width=1pt,color=ffzzcc] (-11.1,9.9)-- (-11.1,10.1);
\draw [line width=1pt,color=ffwwqq] (-14.2,13.2)-- (-8.8,13.2);
\draw [line width=1pt,color=ffwwqq] (-8.8,13.2)-- (-8.8,12.8);
\draw [line width=1pt,color=ffwwqq] (-8.8,12.8)-- (-14.2,12.8);
\draw [line width=1pt,color=ffwwqq] (-14.2,12.8)-- (-14.2,13.2);
\draw [line width=1pt,dotted] (-8.2,12.2)-- (-5.8,12.2);
\draw [line width=1pt,dotted] (-5.8,12.2)-- (-5.8,11.8);
\draw [line width=1pt,dotted] (-5.8,11.8)-- (-8.2,11.8);
\draw [line width=1pt,dotted] (-8.2,11.8)-- (-8.2,12.2);
\draw [line width=1pt,dotted] (-11.1,13.1)-- (-9.9,13.1);
\draw [line width=1pt,dotted] (-9.9,13.1)-- (-9.9,12.9);
\draw [line width=1pt,dotted] (-9.9,12.9)-- (-11.1,12.9);
\draw [line width=1pt,dotted] (-11.1,12.9)-- (-11.1,13.1);
\draw [line width=1pt,color=ffzzcc] (-11.1,12.1)-- (-9.9,12.1);
\draw [line width=1pt,color=ffzzcc] (-9.9,12.1)-- (-9.9,11.9);
\draw [line width=1pt,color=ffzzcc] (-9.9,11.9)-- (-11.1,11.9);
\draw [line width=1pt,color=ffzzcc] (-11.1,11.9)-- (-11.1,12.1);
\draw [line width=1pt,color=ffffww] (-8.2,11.2)-- (-5.8,11.2);
\draw [line width=1pt,color=ffffww] (-5.8,11.2)-- (-5.8,10.8);
\draw [line width=1pt,color=ffffww] (-5.8,10.8)-- (-8.2,10.8);
\draw [line width=1pt,color=ffffww] (-8.2,10.8)-- (-8.2,11.2);
\draw [line width=1pt,color=qqqqff] (-11.1,11.1)-- (-9.9,11.1);
\draw [line width=1pt,color=qqqqff] (-9.9,11.1)-- (-9.9,10.9);
\draw [line width=1pt,color=qqqqff] (-9.9,10.9)-- (-11.1,10.9);
\draw [line width=1pt,color=qqqqff] (-11.1,10.9)-- (-11.1,11.1);
\draw [line width=1pt,dotted] (-14.2,22.2)-- (-8.8,22.2);
\draw [line width=1pt,dotted] (-8.8,22.2)-- (-8.8,21.8);
\draw [line width=1pt,dotted] (-8.8,21.8)-- (-14.2,21.8);
\draw [line width=1pt,dotted] (-14.2,21.8)-- (-14.2,22.2);
\draw [line width=1pt,color=qqqqff] (-8.8,23.2)-- (-14.2,23.2);
\draw [line width=1pt,color=qqqqff] (-14.2,23.2)-- (-14.2,22.8);
\draw [line width=1pt,color=qqqqff] (-14.2,22.8)-- (-8.8,22.8);
\draw [line width=1pt,color=qqqqff] (-8.8,22.8)-- (-8.8,23.2);
\draw [line width=1pt,dotted] (-11.1,23.1)-- (-9.9,23.1);
\draw [line width=1pt,dotted] (-9.9,23.1)-- (-9.9,22.9);
\draw [line width=1pt,dotted] (-9.9,22.9)-- (-11.1,22.9);
\draw [line width=1pt,dotted] (-11.1,22.9)-- (-11.1,23.1);
\draw [line width=1pt,color=ffwwqq] (-8.2,22.2)-- (-5.8,22.2);
\draw [line width=1pt,color=ffwwqq] (-5.8,22.2)-- (-5.8,21.8);
\draw [line width=1pt,color=ffwwqq] (-5.8,21.8)-- (-8.2,21.8);
\draw [line width=1pt,color=ffwwqq] (-8.2,21.8)-- (-8.2,22.2);
\draw [line width=1pt,color=ffffww] (-8.2,20.2)-- (-5.8,20.2);
\draw [line width=1pt,color=ffffww] (-5.8,20.2)-- (-5.8,19.8);
\draw [line width=1pt,color=ffffww] (-5.8,19.8)-- (-8.2,19.8);
\draw [line width=1pt,color=ffffww] (-8.2,19.8)-- (-8.2,20.2);
\draw [line width=1pt,color=ffzzcc] (-11.1,21.1)-- (-9.9,21.1);
\draw [line width=1pt,color=ffzzcc] (-9.9,21.1)-- (-9.9,20.9);
\draw [line width=1pt,color=ffzzcc] (-9.9,20.9)-- (-11.1,20.9);
\draw [line width=1pt,color=ffzzcc] (-11.1,20.9)-- (-11.1,21.1);
\draw [line width=1pt,color=ffxfqq] (-14.2,21.2)-- (-8.8,21.2);
\draw [line width=1pt,color=ffxfqq] (-8.8,21.2)-- (-8.8,20.8);
\draw [line width=1pt,color=ffxfqq] (-8.8,20.8)-- (-14.2,20.8);
\draw [line width=1pt,color=ffxfqq] (-14.2,20.8)-- (-14.2,21.2);
\draw [line width=1pt,color=qqqqff] (-11.1,22.1)-- (-9.9,22.1);
\draw [line width=1pt,color=qqqqff] (-9.9,22.1)-- (-9.9,21.9);
\draw [line width=1pt,color=qqqqff] (-9.9,21.9)-- (-11.1,21.9);
\draw [line width=1pt,color=qqqqff] (-11.1,21.9)-- (-11.1,22.1);
\draw [line width=1pt,color=ccwwqq] (-9.9,20.1)-- (-9.9,19.9);
\draw [line width=1pt,color=qqffqq] (-8.2,22.8)-- (-5.8,22.8);
\draw [line width=1pt,color=ffzzcc] (-8.2,19.2)-- (-5.8,19.2);
\draw [line width=1pt] (-7,18)-- (-6,18);
\draw [line width=1pt] (-7,9)-- (-6,9);
\draw [line width=1pt] (-7,0)-- (-6,0);
\begin{scriptsize}
\draw [fill=ffwwqq] (-14,4) circle (2.5pt);
\draw [fill=ffffff] (-13,4) circle (2.5pt);
\draw [fill=ffwwqq] (-12,4) circle (2.5pt);
\draw [fill=qqqqff] (-11,4) circle (2.5pt);
\draw [fill=ffffff] (-10,4) circle (2.5pt);
\draw [fill=ffffff] (-9,4) circle (2.5pt);
\draw [fill=ffffff] (-8,4) circle (2.5pt);
\draw [fill=ffffff] (-7,4) circle (2.5pt);
\draw [fill=qqffqq] (-7,3) circle (2.5pt);
\draw [fill=ffffff] (-8,3) circle (2.5pt);
\draw [fill=ffffff] (-9,3) circle (2.5pt);
\draw [fill=ffffff] (-10,3) circle (2.5pt);
\draw [fill=ffffff] (-11,3) circle (2.5pt);
\draw [fill=ffffff] (-12,3) circle (2.5pt);
\draw [fill=ffffff] (-13,3) circle (2.5pt);
\draw [fill=ffffff] (-14,3) circle (2.5pt);
\draw [fill=ffffff] (-14,2) circle (2.5pt);
\draw [fill=qqqqff] (-13,2) circle (2.5pt);
\draw [fill=ffffff] (-12,2) circle (2.5pt);
\draw [fill=ffffff] (-11,2) circle (2.5pt);
\draw [fill=ffffff] (-10,2) circle (2.5pt);
\draw [fill=ffffff] (-9,2) circle (2.5pt);
\draw [fill=ffffff] (-8,2) circle (2.5pt);
\draw [fill=ffwwqq] (-7,2) circle (2.5pt);
\draw [fill=qqffqq] (-7,1) circle (2.5pt);
\draw [fill=ffffff] (-8,1) circle (2.5pt);
\draw [fill=ffffff] (-9,1) circle (2.5pt);
\draw [fill=ffffff] (-10,1) circle (2.5pt);
\draw [fill=ffffff] (-11,1) circle (2.5pt);
\draw [fill=ffffff] (-12,1) circle (2.5pt);
\draw [fill=ffffff] (-13,1) circle (2.5pt);
\draw [fill=ffffff] (-14,1) circle (2.5pt);
\draw [fill=ffffff] (-14,10) circle (2.5pt);
\draw [fill=ffffff] (-13,10) circle (2.5pt);
\draw [fill=ffffff] (-12,10) circle (2.5pt);
\draw [fill=ffffff] (-11,10) circle (2.5pt);
\draw [fill=ffffff] (-10,10) circle (2.5pt);
\draw [fill=qqffqq] (-9,10) circle (2.5pt);
\draw [fill=ffffff] (-8,10) circle (2.5pt);
\draw [fill=ffffff] (-7,10) circle (2.5pt);
\draw [fill=ffffff] (-7,11) circle (2.5pt);
\draw [fill=ffffff] (-7,12) circle (2.5pt);
\draw [fill=ffffff] (-7,13) circle (2.5pt);
\draw [fill=ffffww] (-8,13) circle (2.5pt);
\draw [fill=ffffff] (-8,12) circle (2.5pt);
\draw [fill=ffffff] (-8,11) circle (2.5pt);
\draw [fill=ffffff] (-9,11) circle (2.5pt);
\draw [fill=ffwwqq] (-9,12) circle (2.5pt);
\draw [fill=ffffff] (-9,13) circle (2.5pt);
\draw [fill=ffffff] (-10,13) circle (2.5pt);
\draw [fill=ffffff] (-10,12) circle (2.5pt);
\draw [fill=ffffff] (-10,11) circle (2.5pt);
\draw [fill=qqqqff] (-11,11) circle (2.5pt);
\draw [fill=ffffff] (-11,12) circle (2.5pt);
\draw [fill=ffffff] (-11,13) circle (2.5pt);
\draw [fill=ffffff] (-12,13) circle (2.5pt);
\draw [fill=ffffff] (-12,12) circle (2.5pt);
\draw [fill=ffffff] (-12,11) circle (2.5pt);
\draw [fill=ffffff] (-13,11) circle (2.5pt);
\draw [fill=ffffff] (-13,12) circle (2.5pt);
\draw [fill=ffffff] (-13,13) circle (2.5pt);
\draw [fill=ffwwqq] (-14,13) circle (2.5pt);
\draw [fill=ffffff] (-14,12) circle (2.5pt);
\draw [fill=ffffff] (-14,11) circle (2.5pt);
\draw [fill=cqcqcq] (-14,0) circle (2.5pt);
\draw[color=cqcqcq] (-14,0.3) node {$x_1$};
\draw [fill=cqcqcq] (-13,0) circle (2.5pt);
\draw[color=cqcqcq] (-13,0.3) node {$x_2$};
\draw [fill=cqcqcq] (-12,0) circle (2.5pt);
\draw[color=cqcqcq] (-12,0.3) node {$x_3$};
\draw [fill=cqcqcq] (-11,0) circle (2.5pt);
\draw[color=cqcqcq] (-11,0.3) node {$x_4$};
\draw [fill=cqcqcq] (-10,0) circle (2.5pt);
\draw[color=cqcqcq] (-10,0.3) node {$x_5$};
\draw [fill=cqcqcq] (-9,0) circle (2.5pt);
\draw[color=cqcqcq] (-9,0.3) node {$x_6$};
\draw [fill=cqcqcq] (-8,0) circle (2.5pt);
\draw[color=cqcqcq] (-8,0.3) node {$x_7$};
\draw [fill=cqcqcq] (-7,0) circle (2.5pt);
\draw[color=cqcqcq] (-7,0.3) node {$x_8$};
\draw [fill=cqcqcq] (-15,1) circle (2.5pt);
\draw[color=cqcqcq] (-14.7,1) node {$y_1$};
\draw [fill=cqcqcq] (-15,2) circle (2.5pt);
\draw[color=cqcqcq] (-14.7,2) node {$y_2$};
\draw [fill=cqcqcq] (-15,3) circle (2.5pt);
\draw[color=cqcqcq] (-14.7,3) node {$y_3$};
\draw [fill=cqcqcq] (-15,4) circle (2.5pt);
\draw[color=cqcqcq] (-14.7,4) node {$y_4$};
\draw [fill=cqcqcq] (-14,9) circle (2.5pt);
\draw[color=cqcqcq] (-14,9.3) node {$x_1$};
\draw [fill=cqcqcq] (-13,9) circle (2.5pt);
\draw[color=cqcqcq] (-13,9.3) node {$x_2$};
\draw [fill=cqcqcq] (-12,9) circle (2.5pt);
\draw[color=cqcqcq] (-12,9.3) node {$x_3$};
\draw [fill=cqcqcq] (-11,9) circle (2.5pt);
\draw[color=cqcqcq] (-11,9.3) node {$x_4$};
\draw [fill=cqcqcq] (-10,9) circle (2.5pt);
\draw[color=cqcqcq] (-10,9.3) node {$x_5$};
\draw [fill=cqcqcq] (-9,9) circle (2.5pt);
\draw[color=cqcqcq] (-9,9.3) node {$x_6$};
\draw [fill=cqcqcq] (-8,9) circle (2.5pt);
\draw[color=cqcqcq] (-8,9.3) node {$x_7$};
\draw [fill=cqcqcq] (-7,9) circle (2.5pt);
\draw[color=cqcqcq] (-7,9.3) node {$x_8$};
\draw [fill=cqcqcq] (-15,10) circle (2.5pt);
\draw[color=cqcqcq] (-14.7,10) node {$y_1$};
\draw [fill=cqcqcq] (-15,11) circle (2.5pt);
\draw[color=cqcqcq] (-14.7,11) node {$y_2$};
\draw [fill=cqcqcq] (-15,12) circle (2.5pt);
\draw[color=cqcqcq] (-14.7,12) node {$y_3$};
\draw [fill=cqcqcq] (-15,13) circle (2.5pt);
\draw[color=cqcqcq] (-14.7,13) node {$y_4$};
\draw [fill=cqcqcq] (-2.994015505438793,1.997063168701213) circle (2.5pt);
\draw[color=cqcqcq] (-2.7,2) node {$z_1$};
\draw [fill=cqcqcq] (-3,12) circle (2.5pt);
\draw[color=cqcqcq] (-2.7,12) node {$z_2$};
\draw[color=black] (-2.5,11.139454729106477) node {$Z$};
\draw[color=black] (-15.986329869974245,10) node {$Y$};
\draw[color=black] (-16.037366696596532,1) node {$Y$};
\draw [fill=black] (-14.2,2.8) circle (0.5pt);
\draw [fill=black] (-8.8,2.8) circle (0.5pt);
\draw [fill=black] (-8.8,3.2) circle (0.5pt);
\draw [fill=black] (-14.2,3.2) circle (0.5pt);
\draw [fill=black] (-11.1,4.1) circle (0.5pt);
\draw [fill=black] (-9.9,4.1) circle (0.5pt);
\draw [fill=black] (-9.9,3.9) circle (0.5pt);
\draw [fill=black] (-11.1,3.9) circle (0.5pt);
\draw [fill=black] (-8.2,10.2) circle (0.5pt);
\draw [fill=black] (-5.8,10.2) circle (0.5pt);
\draw [fill=black] (-5.8,9.8) circle (0.5pt);
\draw [fill=black] (-8.2,9.8) circle (0.5pt);
\draw [fill=black] (-11.1,14.1) circle (0.5pt);
\draw [fill=black] (-9.9,14.1) circle (0.5pt);
\draw [fill=black] (-9.9,13.9) circle (0.5pt);
\draw [fill=black] (-11.1,13.9) circle (0.5pt);
\draw [fill=black] (-11.1,3.1) circle (0.5pt);
\draw [fill=black] (-9.9,3.1) circle (0.5pt);
\draw [fill=black] (-9.9,2.9) circle (0.5pt);
\draw [fill=black] (-11.1,2.9) circle (0.5pt);
\draw [fill=black] (-14.2,4.2) circle (0.5pt);
\draw [fill=black] (-8.8,4.2) circle (0.5pt);
\draw [fill=black] (-8.8,3.8) circle (0.5pt);
\draw [fill=black] (-14.2,3.8) circle (0.5pt);
\draw [fill=black] (-14.2,2.2) circle (0.5pt);
\draw [fill=black] (-8.8,2.2) circle (0.5pt);
\draw [fill=black] (-8.8,1.8) circle (0.5pt);
\draw [fill=black] (-14.2,1.8) circle (0.5pt);
\draw [fill=black] (-11.1,2.1) circle (0.5pt);
\draw [fill=black] (-9.9,2.1) circle (0.5pt);
\draw [fill=black] (-9.9,1.9) circle (0.5pt);
\draw [fill=black] (-11.1,1.9) circle (0.5pt);
\draw [fill=black] (-8.2,2.2) circle (0.5pt);
\draw [fill=black] (-5.8,2.2) circle (0.5pt);
\draw [fill=black] (-5.8,1.8) circle (0.5pt);
\draw [fill=black] (-8.2,1.8) circle (0.5pt);
\draw [fill=black] (-11.1,1.1) circle (0.5pt);
\draw [fill=black] (-9.9,1.1) circle (0.5pt);
\draw [fill=black] (-9.9,0.9) circle (0.5pt);
\draw [fill=black] (-11.1,0.9) circle (0.5pt);
\draw [fill=black] (-8.2,1.2) circle (0.5pt);
\draw [fill=black] (-5.8,1.2) circle (0.5pt);
\draw [fill=black] (-5.8,0.8) circle (0.5pt);
\draw [fill=black] (-8.2,0.8) circle (0.5pt);
\draw [fill=black] (-14.2,1.2) circle (0.5pt);
\draw [fill=black] (-8.8,1.2) circle (0.5pt);
\draw [fill=black] (-8.8,0.8) circle (0.5pt);
\draw [fill=black] (-14.2,0.8) circle (0.5pt);
\draw [fill=cqcqcq] (-15,5) circle (2.5pt);
\draw[color=cqcqcq] (-14.7,5) node {$y_5$};
\draw [fill=cqcqcq] (-15,14) circle (2.5pt);
\draw[color=cqcqcq] (-14.7,14) node {$y_5$};
\draw [fill=ffffff] (-14,14) circle (2.5pt);
\draw [fill=ffffff] (-13,14) circle (2.5pt);
\draw [fill=ffffff] (-12,14) circle (2.5pt);
\draw [fill=qqqqff] (-11,14) circle (2.5pt);
\draw [fill=ffffff] (-10,14) circle (2.5pt);
\draw [fill=ffffff] (-9,14) circle (2.5pt);
\draw [fill=ffffff] (-8,14) circle (2.5pt);
\draw [fill=ffffff] (-7,14) circle (2.5pt);
\draw [fill=ffffff] (-14,5) circle (2.5pt);
\draw [fill=ffffff] (-13,5) circle (2.5pt);
\draw [fill=ffffff] (-12,5) circle (2.5pt);
\draw [fill=ffffff] (-11,5) circle (2.5pt);
\draw [fill=ffffff] (-10,5) circle (2.5pt);
\draw [fill=qqffqq] (-9,5) circle (2.5pt);
\draw [fill=ffffff] (-8,5) circle (2.5pt);
\draw [fill=ffffff] (-7,5) circle (2.5pt);
\draw [fill=cqcqcq] (-14,18) circle (2.5pt);
\draw[color=cqcqcq] (-14,18.3) node {$x_1$};
\draw [fill=cqcqcq] (-13,18) circle (2.5pt);
\draw[color=cqcqcq] (-13,18.3) node {$x_2$};
\draw [fill=cqcqcq] (-12,18) circle (2.5pt);
\draw[color=cqcqcq] (-12,18.3) node {$x_3$};
\draw [fill=cqcqcq] (-11,18) circle (2.5pt);
\draw[color=cqcqcq] (-11,18.3) node {$x_4$};
\draw [fill=cqcqcq] (-10,18) circle (2.5pt);
\draw[color=cqcqcq] (-10,18.3) node {$x_5$};
\draw [fill=cqcqcq] (-9,18) circle (2.5pt);
\draw[color=cqcqcq] (-9,18.3) node {$x_6$};
\draw [fill=cqcqcq] (-8,18) circle (2.5pt);
\draw[color=cqcqcq] (-8,18.3) node {$x_7$};
\draw [fill=cqcqcq] (-7,18) circle (2.5pt);
\draw[color=cqcqcq] (-7,18.3) node {$x_8$};
\draw [fill=ffffff] (-14,19) circle (2.5pt);
\draw [fill=ffffff] (-13,19) circle (2.5pt);
\draw [fill=ffxfqq] (-12,19) circle (2.5pt);
\draw [fill=ffffff] (-11,19) circle (2.5pt);
\draw [fill=ffffff] (-10,19) circle (2.5pt);
\draw [fill=ffffff] (-9,19) circle (2.5pt);
\draw [fill=ffffff] (-8,19) circle (2.5pt);
\draw [fill=ffffff] (-7,19) circle (2.5pt);
\draw [fill=ffffww] (-7,20) circle (2.5pt);
\draw [fill=ffffff] (-7,21) circle (2.5pt);
\draw [fill=ffffff] (-7,22) circle (2.5pt);
\draw [fill=qqffqq] (-7,23) circle (2.5pt);
\draw [fill=ffffff] (-8,23) circle (2.5pt);
\draw [fill=ffffff] (-8,22) circle (2.5pt);
\draw [fill=ffffff] (-8,21) circle (2.5pt);
\draw [fill=ffffww] (-8,20) circle (2.5pt);
\draw [fill=ffffff] (-9,20) circle (2.5pt);
\draw [fill=ffffff] (-9,21) circle (2.5pt);
\draw [fill=ffffff] (-9,22) circle (2.5pt);
\draw [fill=ffffff] (-9,23) circle (2.5pt);
\draw [fill=ffffff] (-10,23) circle (2.5pt);
\draw [fill=qqqqff] (-10,22) circle (2.5pt);
\draw [fill=ffffff] (-10,21) circle (2.5pt);
\draw [fill=ffffff] (-10,20) circle (2.5pt);
\draw [fill=ffffff] (-11,20) circle (2.5pt);
\draw [fill=ffffff] (-11,21) circle (2.5pt);
\draw [fill=ffffff] (-11,22) circle (2.5pt);
\draw [fill=ffffff] (-11,23) circle (2.5pt);
\draw [fill=ffffff] (-12,23) circle (2.5pt);
\draw [fill=ffffff] (-12,22) circle (2.5pt);
\draw [fill=ffxfqq] (-12,21) circle (2.5pt);
\draw [fill=ffffff] (-12,20) circle (2.5pt);
\draw [fill=ffffff] (-13,20) circle (2.5pt);
\draw [fill=ffffff] (-13,21) circle (2.5pt);
\draw [fill=ffffff] (-13,22) circle (2.5pt);
\draw [fill=qqqqff] (-13,23) circle (2.5pt);
\draw [fill=ffffff] (-14,23) circle (2.5pt);
\draw [fill=ffffff] (-14,22) circle (2.5pt);
\draw [fill=ffffff] (-14,21) circle (2.5pt);
\draw [fill=cqcqcq] (-15,19) circle (2.5pt);
\draw[color=cqcqcq] (-14.7,19) node {$y_1$};
\draw [fill=cqcqcq] (-15,20) circle (2.5pt);
\draw[color=cqcqcq] (-14.7,20) node {$y_2$};
\draw [fill=cqcqcq] (-15,21) circle (2.5pt);
\draw[color=cqcqcq] (-14.7,21) node {$y_3$};
\draw [fill=cqcqcq] (-15,22) circle (2.5pt);
\draw[color=cqcqcq] (-14.7,22) node {$y_4$};
\draw [fill=cqcqcq] (-15,23) circle (2.5pt);
\draw[color=cqcqcq] (-14.7,23) node {$y_5$};
\draw[color=black] (-15.986329869974245,19.050162855561286) node {$Y$};
\draw [fill=cqcqcq] (-3,21) circle (2.5pt);
\draw[color=cqcqcq] (-2.7,21) node {$z_3$};
\draw [fill=qqffqq] (-14,20) circle (2.5pt);
\draw [fill=black] (-8.2,13.8) circle (0.5pt);
\draw [fill=black] (-5.8,13.8) circle (0.5pt);
\draw [fill=black] (-5.8,14.2) circle (0.5pt);
\draw [fill=black] (-8.2,14.2) circle (0.5pt);
\draw [fill=black] (-8.2,23.2) circle (0.5pt);
\draw [fill=black] (-5.8,23.2) circle (0.5pt);
\draw [fill=black] (-5.8,22.8) circle (0.5pt);
\draw [fill=black] (-8.2,22.8) circle (0.5pt);
\draw [fill=black] (-11.1,19.9) circle (0.5pt);
\draw [fill=black] (-9.9,19.9) circle (0.5pt);
\draw [fill=black] (-9.9,20.1) circle (0.5pt);
\draw [fill=black] (-11.1,20.1) circle (0.5pt);
\draw [fill=black] (-8.2,20.8) circle (0.5pt);
\draw [fill=black] (-5.8,20.8) circle (0.5pt);
\draw [fill=black] (-5.8,21.2) circle (0.5pt);
\draw [fill=black] (-8.2,21.2) circle (0.5pt);
\draw [fill=black] (-8.2,18.8) circle (0.5pt);
\draw [fill=black] (-5.8,18.8) circle (0.5pt);
\draw [fill=black] (-5.8,19.2) circle (0.5pt);
\draw [fill=black] (-8.2,19.2) circle (0.5pt);
\draw [fill=black] (-14.2,20.2) circle (0.5pt);
\draw [fill=black] (-8.8,20.2) circle (0.5pt);
\draw [fill=black] (-8.8,19.8) circle (0.5pt);
\draw [fill=black] (-14.2,19.8) circle (0.5pt);
\draw [fill=black] (-8.8,19.2) circle (0.5pt);
\draw [fill=black] (-14.2,19.2) circle (0.5pt);
\draw [fill=black] (-14.2,18.8) circle (0.5pt);
\draw [fill=black] (-8.8,18.8) circle (0.5pt);
\draw [fill=black] (-11.1,19.1) circle (0.5pt);
\draw [fill=black] (-9.9,19.1) circle (0.5pt);
\draw [fill=black] (-9.9,18.9) circle (0.5pt);
\draw [fill=black] (-11.1,18.9) circle (0.5pt);
\draw[color=black] (-12.668936139525536,-1.6197519264657936) node {$X$};
\draw[color=black] (-12.924120272636976,7.8730978252799755) node {$X$};
\draw[color=black] (-12.873083446014688,16.753505657558275) node {$X$};
\draw [fill=black] (-8.2,2.8) circle (0.5pt);
\draw [fill=black] (-8.2,3.2) circle (0.5pt);
\draw [fill=black] (-5.8,3.2) circle (0.5pt);
\draw [fill=black] (-5.8,2.8) circle (0.5pt);
\draw [fill=black] (-8.2,12.8) circle (0.5pt);
\draw [fill=black] (-8.2,13.2) circle (0.5pt);
\draw [fill=black] (-8.2,5.2) circle (0.5pt);
\draw [fill=black] (-5.8,5.2) circle (0.5pt);
\draw [fill=black] (-5.8,4.8) circle (0.5pt);
\draw [fill=black] (-8.2,4.8) circle (0.5pt);
\draw [fill=black] (-8.2,4.2) circle (0.5pt);
\draw [fill=black] (-5.8,4.2) circle (0.5pt);
\draw [fill=black] (-5.8,3.8) circle (0.5pt);
\draw [fill=black] (-8.2,3.8) circle (0.5pt);
\draw [fill=black] (-11.1,5.1) circle (0.5pt);
\draw [fill=black] (-9.9,5.1) circle (0.5pt);
\draw [fill=black] (-9.9,4.9) circle (0.5pt);
\draw [fill=black] (-11.1,4.9) circle (0.5pt);
\draw [fill=black] (-14.2,5.2) circle (0.5pt);
\draw [fill=black] (-8.8,5.2) circle (0.5pt);
\draw [fill=black] (-8.8,4.8) circle (0.5pt);
\draw [fill=black] (-14.2,4.8) circle (0.5pt);
\draw [fill=black] (-5.8,13.2) circle (0.5pt);
\draw [fill=black] (-5.8,12.8) circle (0.5pt);
\draw [fill=black] (-14.2,14.2) circle (0.5pt);
\draw [fill=black] (-8.8,14.2) circle (0.5pt);
\draw [fill=black] (-8.8,13.8) circle (0.5pt);
\draw [fill=black] (-14.2,13.8) circle (0.5pt);
\draw [fill=black] (-14.2,12.2) circle (0.5pt);
\draw [fill=black] (-8.8,12.2) circle (0.5pt);
\draw [fill=black] (-8.8,11.8) circle (0.5pt);
\draw [fill=black] (-14.2,11.8) circle (0.5pt);
\draw [fill=black] (-14.2,11.2) circle (0.5pt);
\draw [fill=black] (-8.8,11.2) circle (0.5pt);
\draw [fill=black] (-8.8,10.8) circle (0.5pt);
\draw [fill=black] (-14.2,10.8) circle (0.5pt);
\draw [fill=black] (-14.2,10.2) circle (0.5pt);
\draw [fill=black] (-8.8,10.2) circle (0.5pt);
\draw [fill=black] (-8.8,9.8) circle (0.5pt);
\draw [fill=black] (-14.2,9.8) circle (0.5pt);
\draw [fill=black] (-11.1,10.1) circle (0.5pt);
\draw [fill=black] (-9.9,10.1) circle (0.5pt);
\draw [fill=black] (-9.9,9.9) circle (0.5pt);
\draw [fill=black] (-11.1,9.9) circle (0.5pt);
\draw [fill=black] (-14.2,13.2) circle (0.5pt);
\draw [fill=black] (-8.8,13.2) circle (0.5pt);
\draw [fill=black] (-8.8,12.8) circle (0.5pt);
\draw [fill=black] (-14.2,12.8) circle (0.5pt);
\draw [fill=black] (-8.2,12.2) circle (0.5pt);
\draw [fill=black] (-5.8,12.2) circle (0.5pt);
\draw [fill=black] (-5.8,11.8) circle (0.5pt);
\draw [fill=black] (-8.2,11.8) circle (0.5pt);
\draw [fill=black] (-11.1,13.1) circle (0.5pt);
\draw [fill=black] (-9.9,13.1) circle (0.5pt);
\draw [fill=black] (-9.9,12.9) circle (0.5pt);
\draw [fill=black] (-11.1,12.9) circle (0.5pt);
\draw [fill=black] (-11.1,12.1) circle (0.5pt);
\draw [fill=black] (-9.9,12.1) circle (0.5pt);
\draw [fill=black] (-9.9,11.9) circle (0.5pt);
\draw [fill=black] (-11.1,11.9) circle (0.5pt);
\draw [fill=black] (-8.2,11.2) circle (0.5pt);
\draw [fill=black] (-5.8,11.2) circle (0.5pt);
\draw [fill=black] (-5.8,10.8) circle (0.5pt);
\draw [fill=black] (-8.2,10.8) circle (0.5pt);
\draw [fill=black] (-11.1,11.1) circle (0.5pt);
\draw [fill=black] (-9.9,11.1) circle (0.5pt);
\draw [fill=black] (-9.9,10.9) circle (0.5pt);
\draw [fill=black] (-11.1,10.9) circle (0.5pt);
\draw [fill=black] (-14.2,22.2) circle (0.5pt);
\draw [fill=black] (-8.8,22.2) circle (0.5pt);
\draw [fill=black] (-8.8,21.8) circle (0.5pt);
\draw [fill=black] (-14.2,21.8) circle (0.5pt);
\draw [fill=black] (-8.8,23.2) circle (0.5pt);
\draw [fill=black] (-14.2,23.2) circle (0.5pt);
\draw [fill=black] (-14.2,22.8) circle (0.5pt);
\draw [fill=black] (-8.8,22.8) circle (0.5pt);
\draw [fill=black] (-11.1,23.1) circle (0.5pt);
\draw [fill=black] (-9.9,23.1) circle (0.5pt);
\draw [fill=black] (-9.9,22.9) circle (0.5pt);
\draw [fill=black] (-11.1,22.9) circle (0.5pt);
\draw [fill=black] (-8.2,22.2) circle (0.5pt);
\draw [fill=black] (-5.8,22.2) circle (0.5pt);
\draw [fill=black] (-5.8,21.8) circle (0.5pt);
\draw [fill=black] (-8.2,21.8) circle (0.5pt);
\draw [fill=black] (-8.2,20.2) circle (0.5pt);
\draw [fill=black] (-5.8,20.2) circle (0.5pt);
\draw [fill=black] (-5.8,19.8) circle (0.5pt);
\draw [fill=black] (-8.2,19.8) circle (0.5pt);
\draw [fill=black] (-11.1,21.1) circle (0.5pt);
\draw [fill=black] (-9.9,21.1) circle (0.5pt);
\draw [fill=black] (-9.9,20.9) circle (0.5pt);
\draw [fill=black] (-11.1,20.9) circle (0.5pt);
\draw [fill=black] (-14.2,21.2) circle (0.5pt);
\draw [fill=black] (-8.8,21.2) circle (0.5pt);
\draw [fill=black] (-8.8,20.8) circle (0.5pt);
\draw [fill=black] (-14.2,20.8) circle (0.5pt);
\draw [fill=black] (-11.1,22.1) circle (0.5pt);
\draw [fill=black] (-9.9,22.1) circle (0.5pt);
\draw [fill=black] (-9.9,21.9) circle (0.5pt);
\draw [fill=black] (-11.1,21.9) circle (0.5pt);
\draw [fill=ffffff] (-6,23) circle (2.5pt);
\draw [fill=ffwwqq] (-6,22) circle (2.5pt);
\draw [fill=ffffff] (-6,21) circle (2.5pt);
\draw [fill=ffffff] (-6,20) circle (2.5pt);
\draw [fill=ffffff] (-6,19) circle (2.5pt);
\draw [fill=cqcqcq] (-6,18) circle (2.5pt);
\draw[color=cqcqcq] (-6,18.3) node {$x_9$};
\draw [fill=ffffff] (-6,14) circle (2.5pt);
\draw [fill=ffffff] (-6,13) circle (2.5pt);
\draw [fill=ffffff] (-6,12) circle (2.5pt);
\draw [fill=ffffww] (-6,11) circle (2.5pt);
\draw [fill=ffffff] (-6,10) circle (2.5pt);
\draw [fill=cqcqcq] (-6,9) circle (2.5pt);
\draw[color=cqcqcq] (-6,9.3) node {$x_{9}$};
\draw [fill=qqqqff] (-6,5) circle (2.5pt);
\draw [fill=ffffff] (-6,4) circle (2.5pt);
\draw [fill=ffffff] (-6,3) circle (2.5pt);
\draw [fill=ffffff] (-6,2) circle (2.5pt);
\draw [fill=ffffff] (-6,1) circle (2.5pt);
\draw [fill=cqcqcq] (-6,0) circle (2.5pt);
\draw[color=cqcqcq] (-6,0.3) node {$x_{9}$};
\end{scriptsize}
\end{tikzpicture}
\caption{Cell coloring for the cartesian product $X\square Y\square Z.$}
\label{fig:main fig}
\end{center}
\end{figure}
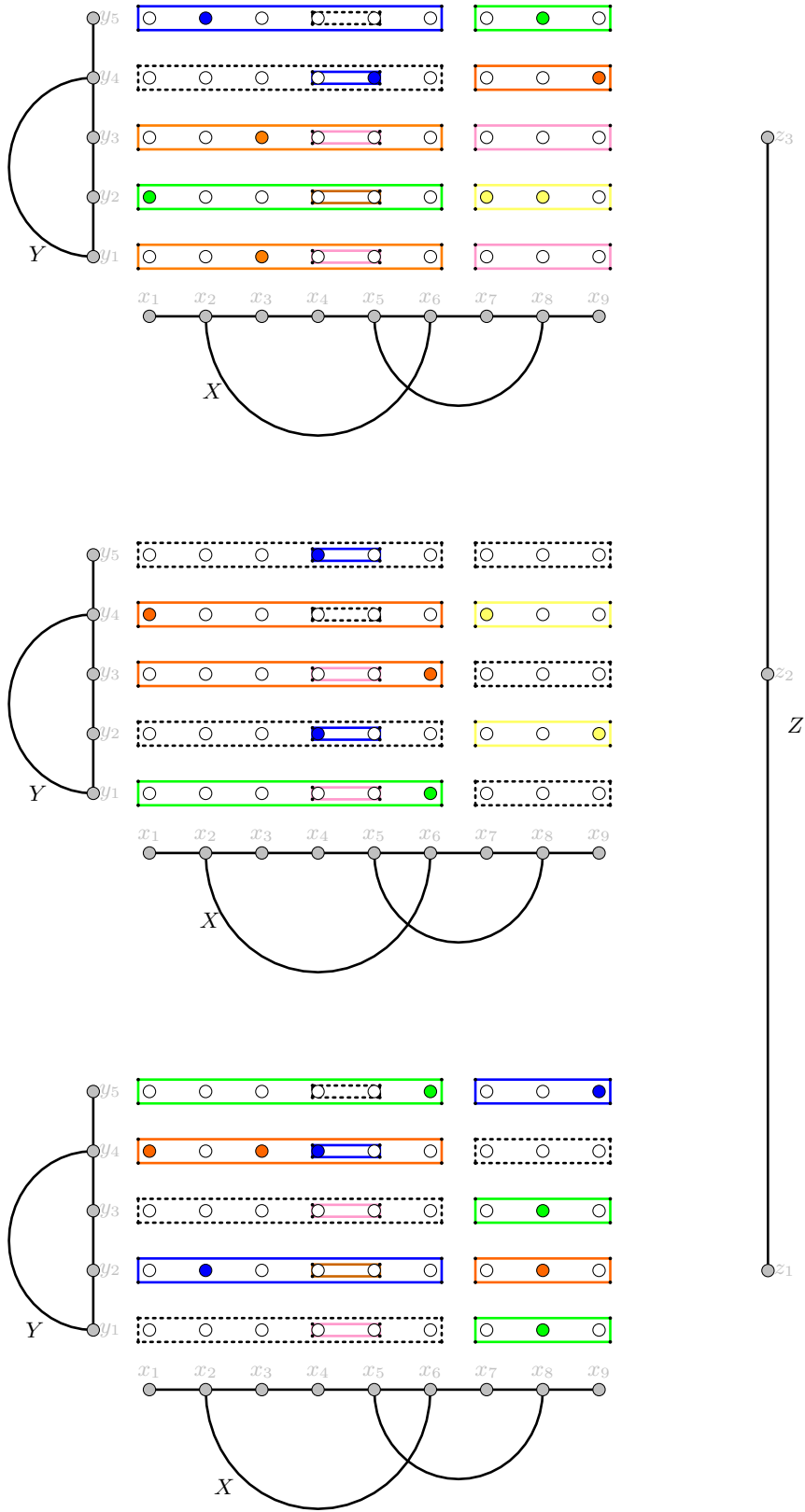
\end{ex}

%\begin{ex} In Example \ref{Main example}, the maximum number of maroon cells is reached in $Z_{2,y_2}.$ The reader is invited to check that this agrees with the graphic representation in Figure \ref{fig:max maroon}.
%\end{ex}

\begin{theoreme}\label{Main Theorem 2} For any pair of graphs $X$ and $Y$ and for any non-negative integer $k,$ we have 
$$\gamma (X\square Y\square P_n)\geq \frac{3}{4}\gamma(P_n)\gamma(X)\gamma(Y),~~~~  \text{if $n=3k$, with $k\geq 1,$}$$
 $$\gamma (X\square Y\square P_n)\geq \frac{3k+1}{4k+2}\gamma(P_n)\gamma(X)\gamma(Y),~~~~  \text{if $n=3k+1$,}$$
 $$\gamma (X\square Y\square P_n)\geq \frac{3k+2}{4k+3}\gamma(P_n)\gamma(X)\gamma(Y),~~~~  \text{if $n=3k+2$.}$$
\end{theoreme}

\begin{proof}
When $Z=P_n,$ combining the two inequalities in Lemma \ref{Main lemma 1} and Lemma \ref{Main lemma 3}, we obtain that for any pair of graphs $X$ and $Y,$ we have
\begin{equation}\label{Eq5}
b+g+y+o+g'+y'+o'+m'\geq n\gamma (X) \gamma (Y) .
\end{equation} 
The inequalities obtained in this theorem result from Equation (\ref{Eq5}) and Lemma \ref{Main lem} and the fact that $\gamma(P_{3k})=k$ if $k\geq 1$ and $\gamma(P_{3k+1})=\gamma(P_{3k+2})=k+1$ for any non-negative integer $k.$ We show below how the first inequality is obtained and leave it to the reader to check the remaining two inequalities. If $n=3k,$ with $k\geq 1,$ by Equation \ref{Eq5} we have:
\begin{equation*}
b+g+y+o+g'+y'+o'+m'\geq 3k\gamma (X) \gamma (Y) .
\end{equation*}
But by Lemma \ref{Main lem}, if $n=3k,$ with $k\geq 1,$ we have
\begin{equation*}
2\gamma (X\square Y\square P_n)\geq m'.
\end{equation*}
Using these two inequalities and the fact that $\gamma (X\square Y\square P_n)=|D|=b+g+y+o\geq b'+g'+y'+o',$ we get:
\begin{equation*}
4\gamma (X\square Y\square P_n)\geq 3k\gamma (X) \gamma (Y).
\end{equation*}
The first inequality of Theorem \ref{Main Theorem} is then obtained using the fact that $\gamma(P_{3k})=k.$
\end{proof}

\end{document}